\documentclass[reqno]{amsart}
\usepackage{amsmath,amssymb}
\usepackage{graphics,epsfig,color}
\usepackage[mathscr]{euscript}
\usepackage{mathtools} 
\newtheorem{theorem}{Theorem}[section]
\newtheorem{proposition}[theorem]{Proposition}
\newtheorem{lemma}[theorem]{Lemma}
\newtheorem{corollary}[theorem]{Corollary}
\theoremstyle{definition}

\newtheorem{assumption}[theorem]{Assumption}

\numberwithin{equation}{section}
\numberwithin{theorem}{section}
\renewcommand{\epsilon}{\varepsilon}
\newcommand{\ve}{\varepsilon}
\newcommand{\mc}[1]{{\mathcal #1}}
\newcommand{\bb}[1]{{\mathbb #1}}
\newcommand{\ms}[1]{{\mathscr #1}}
\newcommand{\bs}[1]{{\boldsymbol #1}}

\newcommand{\eps}{\varepsilon}
\newcommand{\e}{\epsilon}
\newcommand{\ee}{\mathrm{e}}

\newcommand{\opid}{\mathop  {\rm id}\nolimits}

\newcommand{\Ent}{\mathop{\rm Ent}\nolimits}

\newcommand{\de}{\mathop{}\!\mathrm{d}}
\newcommand{\dhedit}[1]{\textcolor{black}{#1}}

\title[Dynamical phase transitions for the Boltzmann equation]
{Large time analysis of the rate function associated to the Boltzmann equation:\\
   dynamical phase transitions}
\author[G.\ Basile]{Giada Basile}
\address{Giada Basile \hfill\break \indent
   Dipartimento di Matematica, Universit\`a di Roma `La Sapienza'
   \hfill\break \indent
   P.le Aldo Moro 2, 00185 Roma, Italy}
 \email{basile@mat.uniroma1.it}
 \author[D.\ Benedetto]{Dario Benedetto}
 \address{Dario Benedetto \hfill\break \indent
   Dipartimento di Matematica, Universit\`a di Roma `La Sapienza'
   \hfill\break \indent
   P.le Aldo Moro 2, 00185 Roma, Italy}
 \email{benedetto@mat.uniroma1.it}
\author[L.\ Bertini]{Lorenzo Bertini}
\address{Lorenzo Bertini \hfill\break \indent
   Dipartimento di Matematica, Universit\`a di Roma `La Sapienza'
   \hfill\break \indent
   P.le Aldo Moro 2, 00185 Roma, Italy}
 \email{bertini@mat.uniroma1.it}
  \author[D.\ Heydecker]{Daniel Heydecker}
  \address{Department of Mathematics, Huxley Building, 180 Queen's Gate, South Kensington, London SW7 2AZ}
   \email{d.heydecker@imperial.ac.uk}

\begin{document}
\begin{abstract}
We analyse the large time behaviour of the rate function
  that describes the probability of large fluctuations
  of an underlying microscopic model associated to the homogeneous
  Boltzmann equation, such as the Kac walk.
  We consider in particular the asymptotic of the number of collisions,
  per
  particle and per unit of time,
  and show it exhibits a phase transition in the joint limit
  in which the number  of particles $N$ and the time interval $[0,T]$
  diverge. 
  More precisely,
  due to the existence
  of Lu-Wennberg solutions,
  the corresponding limiting rate function
  vanishes for subtypical values of the number of collisions.
  We also analyse the second order large deviations
  showing that  the  probability of subtypical fluctuations
  is exponentially small in $N$, independently on $T$.
  As a  key point, we establish the controllability of
  the homogeneous Boltzmann equation.
\end{abstract}



\maketitle
\thispagestyle{empty}

\section{Introduction}
\label{s:0}

The limiting behaviour of many body systems in the low density regime
is described by kinetic equations. We focus on the spatially homogeneous case,
for which the typical behaviour is described by 
the homogeneous Boltzmann equation for the
one-particle velocity distribution. 
Recently, some progress has been achieved in the analysis of atypical
behaviour, both in the contest of deterministic and stochastic microscopic
dynamics. In particular, after \cite{Le,Re},  
large deviations results over a finite time window
in the limit of infinite many particles have been proven in
\cite{BBBC1,BBBC2,BBBO,BGSS2,He}.

We focus on the behaviour of the number of collisions
in the joint limit in which the number of particles $N$ and the time
window $[0,T]$ diverge. More precisely, let 
$q_{N,T}$ be the number  of collisions per particle and per unit of time.
As follows from the Boltzmann-Grad limit 
and the large time behaviour of the energy conserving solutions
to the homogeneous Boltzmann equation,
if we let first $N$ and then $T$ diverge,
then $q_{N,T}$ converges to the typical value 
$\bar q_e = \frac 12 \int B(v-v_*,\omega)
M_e(\de v) M_e(\de v_*) \de \omega$,
where $M_e$ is the Maxwellian of energy $e$, prescribed
by the initial conditions, 
$\de \omega$ is the Haar measure on the sphere
and $B$ is the collision kernel.
By ergodicity of the microscopic dynamics, the same limit is obtained
when we let first $T$ and then $N$ diverge.
Note that if $B$ does not depend on the velocities,
as in the case of Maxwellian molecules, then
$\bar q_e$ does not depend on $e$, and in fact
the statistic of $q_{N,T}$ is described
by a Poisson random variable.

By considering first the limit in $N$ and then in $T$, 
the probability of fluctuations of $q_{N,T}$
can be obtained by analysing the large time behaviour
of the rate function associated to the microscopic dynamics.
Here we perform this analysis for the microcanonical ensemble,
in the case of hard sphere interaction,
\dhedit{for which a candidate rate function has been identified in} \cite{BBBC2,He}.
A special role is played by the so-called Lu-Wennberg solutions
to the homogeneous Boltzmann equation \cite{LuW},
which are 
weak solutions  with increasing
energy. A class of these solutions can be obtained form the underlying microscopic dynamics by considering
only fluctuations of the initial distribution and therefore
their cost does not depend on the time interval.
Hence, the probability of a subtypical fluctuations of $q_{N,T}$
is exponentially small in $N$ but independent on $T$.
In contrast, to construct supertypical fluctuations,
it is necessary to change the dynamics and therefore the corresponding
probability is exponentially small in $NT$. Denoting by $\bb P_e^N$ the distribution induced
by the microscopical dynamic with energy per particle given by $e$, we show that
$$
\bb P_e^N ( q_{N,T} \approx q) \sim \ee^{-NT i_e(q)},
$$
where the rate function $i_e$ vanishes in $[0,\bar q_e]$.
We do not obtain a explicit expression for $i_e$, but we provide
explicit 
upper and lower bounds.
The upper bound is obtained by choosing a time independent path,
where  the single time velocity distribution is characterised by a
static variational problem, whose solution is not Maxwellian.
The lower bound is instead obtained by a comparison with a suitable
static strategy.
We expect that neither the upper nor the lower bound is optimal, and in fact
the optimal path should exhibit a non-trivial time dependence.

The second order asymptotic can be formalised
as 
$$
\bb P_e^N ( q_{N,T} \approx q) \sim \ee^{-N j_e(q)}.
$$
Clearly, $j_e(q)=+\infty$ for $q>\bar q_e$.
In the case in which the initial microscopic distribution is the uniform
measure on the energy surface, we obtain a explicit expression for $j_e$,
related to the relative entropy between two Maxwellian.

The present analysis requires two auxiliary results of independent interest
which we next briefly describe.
Given a path with finite rate function, 
we prove the  chain rule for its entropy.
While for energy conserving path this statement has been proven in
\cite{Er2}, we here show 
that it holds
also for path for which the energy is not constant, as in the case
of Lu-Wennberg solutions. 
The second result regards the controllability of
the homogeneous Boltzmann equation.
More precisely,  given two distributions with bounded
energy and entropy,
we show that  they can be connected by a path with finite cost.

In the context of hydrodynamic scaling limits, an analogous problem 
to the one we here  consider is the fluctuation of the total current,
which exhibits interesting behaviour \cite{BGL2,BDGJL,BD}.
The presence of two different scaling regimes for subtypical and
supertypical fluctuation of the total number of jumps has been proven for the so-called
east model, see \cite{BLT,BT}.
We finally refer to \cite{GR} for the asymptotic of the total number of
jumps in the context of non-linear Markov processes.

\section{Background and results}
\label{s:1}

We first recall the analysis in \cite{BBBC2,He} which describes the
large deviations asymptotics for the empirical measure and flow of the
Kac walk over a fixed time interval $[0,T]$ in the limit of infinitely
many particles, $N\to \infty$.
This result yields, by projection, the large deviations
principle of the total number of collisions per particle in this
limit. The corresponding rate functional is expressed by a time
dependent variational problem. By analysing this variational problem
in the limit $T\to\infty$, we then show the total number of collisions
per particle and per unit of time exhibits a dynamical phase
transition in the joint limit $N,T\to \infty$.

\subsection*{Microscopic model and empirical observables}
Fix $d\ge 2$ and set $\Sigma^N = \big(\bb R^d\big)^N$.  We consider
the Kac walk given by the Markov process on the configuration space
$\Sigma^N$ whose generator acts on bounded continuous functions
$f\colon \Sigma^N\to \bb R$ as
\begin{equation*}
  \mathcal L_N f(\bs v)=\frac 1 N \sum_{\{i, j\}} L_{i,j} f (\bs v),
\end{equation*}  
where the sum is carried over the unordered pairs
$\{i, j\}\subset \{1,.., N\}$, $i\neq j$, and
\begin{equation*}
  L_{i,j} f(\bs v) =
  \int_{\bb S_{d-1}}\!\! \de \omega \,B(v_i- v_j, \omega)\big[f
  \big(T^{\omega}_{i,j} \bs v\big ) -f(\bs v)   \big]. 
\end{equation*}  
Here $\bb S_{d-1}$ is the sphere in $\bb R^d$, the post-collisional vector of velocities is given by
\begin{equation}
  \label{rules}
  \big(T^{\omega}_{i,j} \bs v\big )_k = \begin{cases}
    v_i + (\omega \cdot (v_j-v_i))\omega  & \textrm{if } k=i\\
    v_j - (\omega \cdot (v_j-v_i))\omega  & \textrm{if } k=j\\
    v_k & \textrm{otherwise},
    \end{cases}
\end{equation}  
and the collision kernel $B$ is given by
\begin{equation}\label{eq:B}
  B(v-v_*, \omega)=\frac 12 |(v-v_*)\cdot \omega|.
\end{equation}

\dhedit{The collisional dynamics preserves the total particle number, momentum and energy, given by } the integrals of  $$\bs \zeta\colon \bb R^d\mapsto [0, +\infty)\times
\bb R^d, \qquad\bs \zeta =(\zeta_0,
\zeta)(v)=(|v|^2/2, v).$$
In the sequel, we fix $e\in(0,\infty)$ and consider the restriction of the Kac
walk to the set
\begin{equation}\label{sig}
  \Sigma^N_{e}\coloneqq \Big\{\bs v \in\Sigma^N \colon\,
  \frac 1 N\sum_{i=1}^N \bs \zeta(v_i) = (e,0)
  \Big\}. 
\end{equation}
By Gallilean invariance, the choice of vanishing total
momentum is not special, and any other choice can achieved by selecting a suitable frame of
reference. In constrast, the parameter $e$, which represent the energy per particle,
will play an important role in our analysis.  We denote by
$(\bs v(t))_{t\geq 0}$ the Markov process generated by $\mathcal L_N$
which is, by direct computation, ergodic and reversible with respect
to the uniform probability on $\Sigma^N_e$.
Given $T>0$ and a probability $\nu$ on $\Sigma^N_e$, \dhedit{the dynamics carries the  underlying probability measure to the law of the Kac process } $\bb P_\nu^N$ on the Skorokhod
space $D([0,T];\Sigma^N_e)$.

Let $\ms P(\bb R^d)$ be the set of probability measures $\pi$ on
$\bb R^d$ equipped with the weak topology.  We denote by
$\ms P_e$ the subset of $\ms P(\bb R^d)$ given by the
probabilities with vanishing mean and second moment bounded by $2e$;
namely, the set of $\pi\in \ms P(\bb R^d)$ such that
$\pi(\zeta_0)\le e$ and $\pi(\zeta)=0$. $\ms P_e$
is a compact convex subset of $\ms P(\bb R^d)$, and we equip it with
the relative topology and the corresponding Borel $\sigma$-algebra.
For $T>0$, let $D\big([0,T]; \ms P_e\big)$ be the space of
$\ms P_e$-valued c{\`a}dl{\`a}g paths endowed with the
Skorokhod topology and the corresponding Borel $\sigma$-algebra.
The \emph{empirical measure} is the map $\pi^N \colon \Sigma^N_e \to
\ms P_e$ defined by 
\begin{equation}
  \label{1}
  \pi^N(\bs v)\coloneqq \frac 1 N \sum_{i=1}^N \delta_{v_i}.
\end{equation}
We denote by $\bs \pi^N$ the map from 
$D\big([0,T]; \Sigma^N_e \big)$ to $D\big([0,T]; \ms P_e\big)$
defined by $\bs \pi^N_t(\bs v)\coloneqq \pi^N(\bs v(t))$, $t\in [0,T]$.
Under suitable assumptions on the initial conditions, as $N$ tends to infinity\dhedit{, by propagation of chaos \cite{Sn},
the family $\bs \pi^N(\bs v)$ converges to the unique energy-conserving solution to the spatially homogeneous Boltzmann equation}
\begin{equation}
  \label{hbe}
  \partial_t f_t(v) =
  \int_{\bb R^d}  \de v_* \! \int_{S_{d-1}} \de \omega \,
  B(v-v_*, \omega) \big( f_t(v') f_t(v_*')  - f_t(v) f_t(v_*) \big).
\end{equation}

For a fixed time horizon $T>0$, we denote by $\ms M_T$ the  subset of the finite measures $\bs Q$  on
$[0,T]\times \bb R^{2d}\times \bb R^{2d}$  that satisfy
$\bs Q(\de t; \de v,\de v_*, \de v',\de v_*')=
\bs Q(\de t; \de v_*,\de v, \de v',\de v_*') =
\bs Q(\de t; \de v,\de v_*, \de v'_*,\de v')$, which we endow with the weak topology\footnote{i.e. the topology generated by the maps $Q \mapsto Q(F), F\in C_{\mathrm b}([0,T]\times \bb R^{2d}\times \bb R^{2d})$.} and
the corresponding Borel $\sigma-$algebra. 
The \emph{empirical flow} is the map $\bs Q^N_{[0,T]}
\colon D\big([0,T]; \Sigma^N_e \big) \to \ms M_T$
defined by specifying, for each bounded and continuous function $F\colon [0,T]\times \bb R^{2d}\times \bb R^{2d}\to \bb R$
satisfying $$F(t; v, v_*, v',v_*') =F(t;  v_*, v, v',v_*') = F(t;  v, v_*, v'_*,v')$$ the integral
\begin{equation}
  \label{2}
  \bs Q^N_{[0,T]}(\bs v) (F) \coloneqq \frac 1N
  \sum_{\{i,j\}} \sum_{k\ge 1} F\big(\tau^{i,j}_k;
  v_i(\tau^{i,j}_k-),v_j({\tau^{i,j}_k}-),
  v_i(\tau^{i,j}_k),v_j(\tau^{i,j}_k)\big) 
  \quad 
\end{equation}
where $(\tau^{i,j}_k)_{k\ge 1}$ are the
jump times of the pair $(v_i,v_j)$ in the time window $[0,T]$, and $v_i(t-) = \lim_{s\uparrow t} v_i(s)$ is the left-limit.
In view of the conservation of the energy and momentum, the
measure $\bs Q^N_{[0,T]}(\de t;\cdot)$ is supported on the set of pre- and post-collisional velocities satisfying
$$\big\{\bs \zeta(v)+\bs \zeta(v_*)=\bs \zeta(v')+\bs \zeta(v'_*)\big\} 
\subset \bb R^{2d}\times \bb R^{2d}.$$

\subsection*{The rate function in a fixed time window}

For $T>0$, let $\ms S_{e,T}$ be the  subset of
$D\big([0,T]; \ms P_e\big)\times \ms M_T$ given by elements
$(\bs \pi,\bs Q)$ that satisfies,
for each $\phi\in C^1_{\rm{b}}([0,T]\times \bb R^d)$,
the balance equation
\begin{equation}
  \label{bal}
  \pi_T(\phi_T)-\pi_0(\phi_0)-\int_0^T\! \de t\, 
  \pi_t(\partial_t \phi_t)
  -\int \bs Q(\de t;\de v,\de v_*,\de v',\de v_*') \bar \nabla \phi_t
   (v,v_*,v',v'_*)
\end{equation}
where
$$\bar \nabla \phi (v,v_*,v',v'_*) = \phi(v')+\phi(v_*') - \phi(v) -\phi(v_*).$$
By conservation of
the number of particles, for each $\bs v \in \Sigma^N_e$, 
$\bb P^N_{\bs v}$-almost surely, the pair $(\pi^N,\bs Q^N_{[0,T]})$
belongs to the set $\ms S_{e,T}$.

The analysis in \cite{BBBC2} provides the large deviation principle
for the pair $\big(\pi^N, \bs Q^N_{[0,T]}\big)\in \ms S_{e,T}$ in the limit $N\to \infty$ with $T$ fixed.  In order to describe the corresponding rate
function, which takes into account both the fluctuations due to the
initial distribution of the velocities and ones due to the stochastic
dynamics, we \dhedit{must first specify} the initial
distribution for the Kac walk. To this end, fix a probability
$m\in \mc P(\bb R^d)$ satisfying the following conditions.

\begin{assumption}\label{ass:2}
  There exists $\gamma_0^*\in (0, +\infty]$ such that
  \begin{itemize}
  \item[(i)] $m$ is absolutely continuous with respect to the Lebesgue
    measure and $m$ is strictly positive on open sets;
  \item[(ii)] $m(\ee^{\gamma_0 \zeta_0}) < +\infty$ for any
    $\gamma_0\in (-\infty, \gamma_0^*)$, and
    $\lim_{\gamma_0\uparrow \gamma_0^*} m(\ee^{\gamma_0
      \zeta_0})=+\infty$;
  \item[(iii)] for each
    $\bs \gamma=(\gamma_0, \gamma)\in (-\infty, \gamma_0^*)\times \bb
    R^d$ the Fourier transform of
    $\frac {\de m}{\de v}\ee^{\bs\gamma \cdot \bs\zeta}$ belongs to
    $L^1(\bb R^d)$;
  \item[(iv)]
    there exists $c>0$ such that
    $\frac {\de m}{\de v} \ge \frac 1c\exp\{-c  |v|^2\}$.
  \end{itemize}
\end{assumption}

These conditions hold for the most important case when $m=M_e$, the Maxwellian with
vanishing average velocity and average energy $e$, with $\gamma_0^*= d/(2e)$.

Following \cite{BBBC2}, we then consider the Kac
walk with initial distribution of the velocities given by $\nu^N_e$,
the product measure $m^{\otimes N}$ conditioned on vanishing total
momentum and energy per particle given by $e$. Such a conditional measure is well defined in view of the existence of
a regular version of conditional probabilities and $\nu^N_e$ is
supported on $\Sigma^N_e$. Furthermore, by standard properties of
Gaussian measures, if $m=M_e$ then $\nu^N_e$ is the uniform
probability on $\Sigma^N_e$, which is a reversible invariant measure for the Kac process.

For notation convenience we will hereafter also assume that $m$ has
vanishing average momentum and average energy $e$; namely
$m(\bs \zeta)=(e,0)$. This can be achieved by a suitable exponential
tilt which does not affect the conditional probability $\nu^N_e$.

For the choice $\nu^N_e$ of the initial distribution of the
velocities, the static rate function $H_e(\cdot|m) 
\colon \ms P_e\to [0,+\infty]$
is the convex functional given by
\begin{equation}
  \label{srf}
  H_e(\pi|m) \coloneqq  \Ent(\pi\vert m )
  + \gamma_0^*\big[e-\pi\big(\zeta_0\big)\big] 
\end{equation}
where $ \Ent(\cdot \vert \cdot )$ is the relative entropy between 
two probability measures.
This rate functions describes the static large
deviations of the empirical measure $\big\{\pi^N(\bs v)\big\}_{N\ge 1}$
when $\bs v \in \Sigma^N_e$ is sampled according to $\nu^N_e$. 
Note that $H_e(\pi|m)$ can be finite also when the energy of $\pi$ is
strictly smaller than $e$; that is to say that with probability
exponentially small in $N$ -- but not super-exponentially small -- some
of the energy may `escape to infinity'. In the case in which $m=M_e$
so that $\nu^N_e$ is the uniform probability on $\Sigma^N_e$, this
functional has been originally derived in \cite{KR}.
It reads
\begin{equation}
  \label{eq:HH}
  H_e(\pi|M_e) = \int \de \pi \log \frac {\de \pi}{\de v} +
  \frac d2 \Big( \log \frac {4\pi e}d + 1\Big).
\end{equation}

In order to describe the dynamical contribution to the rate function,
let $r(v,v^*;\cdot)$ be the measure on $\bb R^{2d}$ supported on
$\{\bs \zeta(v)+\bs \zeta(v_*)=\bs \zeta(v')+\bs \zeta(v_*')\}$ such
that $r(v,v_*,\de v', \de v_*') = \de \omega \, B(v-v_*,\omega)$,
where $v'$ and $v_*'$ are related to $\omega$ by the collision rules,
as in \eqref{rules}.  For $\pi\in D\big([0,T]; \ms P_e\big)$ let
$\bs Q^\pi\in \ms M_T$ be the measure defined by
\begin{equation}
  \label{4}
  \bs Q^\pi(\de t;\de v,\de v_*,\de v',\de v_*') \coloneqq  \frac 1 2 \de t \, \pi_t(\de v) \pi_t(\de v_*)
  \, r(v,v_*;\de v',\de v_*')
\end{equation}
and observe that $\bs Q^\pi(\de t,\cdot)$ is supported on 
$\big\{\bs \zeta(v)+\bs \zeta(v_*)=\bs \zeta(v')+\bs \zeta(v'_*)\big\}$.

Let $\ms S^{\mathrm{ac}}_{e,T}$ be the subset of $\ms S_{e,T}$ given by the
elements $(\bs \pi,\bs Q)$ such that $\pi\in C\big([0,T];\ms P_e \big)$ and 
$\bs Q\ll \bs Q^\pi$. The dynamical rate function $J_{e,T}
\colon \ms S_{e,T} \to [0,\infty]$ is defined by
\begin{equation}
  \label{5}
  J_{e,T}(\bs \pi,\bs Q)\coloneqq 
  \begin{cases}
    {\displaystyle 
    \int  \de \bs Q^\pi \Big[ 
    \, \frac{\de \bs Q\phantom{^\pi}}{\de \bs Q^\pi} \log \frac{\de
      \bs Q\phantom{^\pi}}{\de \bs Q^\pi} - 
    \Big( \frac{\de \bs Q\phantom{^\pi}}{\de \bs Q^\pi}  -1\Big)\Big]  } &
  \textrm{if } (\bs \pi,\bs Q)\in \ms S^{\mathrm{ac}}_{e,T},\\ \\
    + \infty& \textrm{otherwise.} 
  \end{cases}
\end{equation}
In \cite{BBBC2} it shown that when the initial distribution of the
velocities for the Kac walk is given by $\nu^N_e$ then the pair empirical
measure and flow satisfies a large deviation upper bound with speed
$N$ and rate function $I_{e,T}(\cdot|m)
\colon \ms S_{e,T} \to [0,\infty]$ given by
\begin{equation}
  \label{rft}
  I_{e,T} ((\bs \pi,\bs Q)|m) = H_e(\bs \pi_0|m) + J_{e,T}(\bs \pi,\bs Q).
\end{equation}
where $H_e\colon \ms P_e\to [0,\infty]$ takes into account the
fluctuations in the initial condition while
$J_{e,T}\colon \ms S_{e,T} \to [0,\infty]$ encodes the dynamical
fluctuations.  A matching lower bound is proven for neighborhoods of
pairs $(\bs \pi,\bs Q)$ such that $\bs Q$ has bounded second moment and
for a class of Lu-Wennberg solutions \cite{LuW}.

Denote by $q_{N,T}$ the total number of collisions in the Kac walk per
particle and per unit of time so that $NT q_{N,T}$ is the total number
of collisions in the time window $[0,T]$. From the very definition of
the empirical flow, $q_{N,T}$ can be obtained from the mass of 
$\bs Q^N_{[0,T]}$ and more precisely $q_{N,T} = T^{-1} \bs Q^N_{[0,T]} (1)$.
We claim that $q_{N,T}$ satisfies the law of large number
in probability with respect to $\bb P^N_{\nu^N_e}$
\begin{equation}
  \label{eq:limiti}
  \lim_{T\to +\infty} \lim_{N\to +\infty} q_{N,T} = \bar q_e
  = \lim_{N\to +\infty}\lim_{T\to +\infty} q_{N,T},
\end{equation}
where 
\begin{equation}
  \label{barqe}
  \bar q_e \coloneqq \frac 12
  \int\! M_e(\de v)M_e(\de v_*) B(v-v_*,\omega) \de \omega.
\end{equation}
The first equality in \eqref{eq:limiti}
follows from the convergence of the particle dynamics to
the homogeneous Boltzmann equation with initial datum $m$,
and the convergence of its solution
to the Maxwellian. The second equality follows from the ergodicity of
the microscopic dynamics and the convergence to the Maxwellian
of the empirical measure
when the velocities are sampled  with respect to the
uniform measure on  $\Sigma^N_e$.

For fixed $T>0$, the contraction principle allows us to transfer the large deviation results from the empirical flow to the empirical collision number $\{q_{N,T}\}_{N\ge 1}$ in terms of the rate function $\ms I_{e,T} \colon [0,+\infty)\to
[0,+\infty]$ given by
\begin{equation}
  \label{rfit}
  \ms I_{e,T} (q|m) \coloneqq  \inf\Big\{ I_{e,T} \big((\bs \pi,\bs Q) \big|
  m\big) ,\ 
  (\bs \pi,\bs Q) \in \ms S_{e,T} \textrm{ such that } \bs Q(1) = T q \Big\}.
\end{equation}
It follows from \cite{BBBC2} that $\{q_{N,T}, N\ge 1\}$ satisfies a large deviations upper bound with speed $N$ and this rate,
 but a matching lower bound would require an additional regularity for the
 optimal path for the variational problem on the right had side of
 \eqref{rfit}.
Note that, as already discussed in the
notation, the rate function \eqref{rfit} depends on the choice of the
probability $m$ describing the initial distribution of the velocities.

\subsection*{Main results}


The present purpose is to investigate the large time behaviour of the
rate function in \eqref{rfit} which has a non-trivial structure,
exhibiting in particular two different scaling regimes at large, respectively small, empirical collision numbers. As customary
in large deviations theory \cite{Mar}, the relevant notion to describe
the converge of the functions \eqref{rfit} is the De Giorgi's
$\Gamma$-convergence.

To describe the first scaling regime, we first introduce the limiting
rate function.
Fix $\mu \in \ms P_e$ with $H_e(\mu |M_e )< +\infty$, 
let $i_e(q|\mu)\colon (0,+\infty)\to [0,+\infty]$ be the
function defined by
\begin{equation}
  \label{rfi0}
  i_e(q|\mu) \coloneqq \inf_{T>0} \frac 1T
  \inf_{(\bs \pi,\bs Q)\in \ms A_{e,T}(q|\mu)} J_{e,T} (\bs \pi,\bs Q) 
\end{equation}
where
$$\ms A_{e,T}(q|\mu) = 
\Big\{(\bs \pi,\bs Q) \in \ms S_{e,T}:\, \pi_T = \pi_0 = \mu,  \bs Q(1) = T q \Big\}.
$$
We extend $i_e(\cdot|\mu)$ to a function on $[0,+\infty)$ by setting \begin{equation}
	i_e(0|\mu) := \liminf_{q\downarrow 0} i_e(q|\mu). \label{rfi0'} 
\end{equation}

\begin{proposition}
  \label{prop:ie}
  The function $i_e(\cdot|\mu)$ does not depend on $\mu$, is continuous
  and convex on $[0,+\infty)$.
  Furthermore
  \begin{equation}
    \label{eq:step1}
    i_e(q|\mu) = \lim_{T\to \infty} \frac 1T
    \inf_{(\bs \pi,\bs Q)\in \ms A_{e,T}(q|\mu)} J_{e,T} (\bs \pi,\bs Q).
  \end{equation}
  and vanishes on the interval $[0,\bar q_e]$.
\end{proposition}
In view of this result, here and after we drop the dependence on $\mu$
in the notation for $i_e$.

\begin{theorem}
  \label{t:gc}
  Fix $m\in\ms P_e$ meeting the conditions in Assumption~\ref{ass:2}.
  The sequence of functions $\big\{ T^{-1} \ms I_{e,T}(\cdot |m)
  \big\}_{T>0}$ on $[0,+\infty)$, defined in \eqref{rfit},
  is equi-coercive and
  $\Gamma$-converges to $i_e$ as $T\to\infty$. Namely,  
  \begin{itemize}
  \item [(i)] for each $\ell>0$ there is a compact
    $K_\ell\subset\subset [0,+\infty)$ such that\hfill\break
    $\big\{ q\in [0,+\infty) \colon \ T^{-1}  \ms I_{e,T}
    (\cdot |m) \le \ell
    \big\} \subset K_\ell$ eventually as $T\to \infty$;
  \item [(ii)] for each $q\in[0,+\infty)$ and each sequence $q_T\to q$
    the inequality\hfill\break
    $\liminf_{T\to\infty}  T^{-1}  \ms I_{e,T}( q_T|m) \ge
    i_e(q)$ holds;
  \item [(iii)] for each $q\in[0,+\infty)$ there exists a sequence $q_T\to q$
    such that\hfill\break
    $\limsup_{T\to\infty}  T^{-1}  \ms I_{e,T}( q_T|m) \le i_e(q)$.  
  \end{itemize}
\end{theorem}

In view of standard property of $\Gamma$-convergence, see e.g.\
\cite{Mar}, this statement together with the large deviations upper
bound with fixed $T$ in \cite{BBBC2} implies
the following corollary.
\begin{corollary}
  For each closed set $C\subset [0,+\infty)$
  $$\limsup_{T\to +\infty}\limsup_{N\to +\infty}
  \frac 1{NT} \log \bb P^N_{\nu^N_e} (q_{N,T} \in C)
  \le - \inf_{q\in C} i_e(q).$$
\end{corollary}

According to the arguments in \cite{BGL2}, it should be possible to show
that the same limiting function is obtained when the order of the limit in
$N$ and $T$ is exchanged.

While we are not able to compute the limiting rate function $i_e(q)$
for $q> \bar q_e$, we are able to obtain an upper and a lower bound,
in terms of a ``static strategy''.
Given $\pi\in \ms P_e$, with density $f$,
set
$$ R_2(\pi) =    \frac 12\int\!
ff_* B \de \omega \de v \de v_*,\ \  R_4(\pi) =    \frac 12\int\!
\sqrt{ ff_* f'f_*'}B\de \omega  \de v \de v_*,
$$
in which we denote by $f$, $f_*$, $f'$, $f_*'$ the density  evaluated
in $v$, $v_*$, $v'$, $v_*'$.
Note that, by Cauchy-Schwartz, $R_4(\pi) \le R_2(\pi)$.

\begin{theorem}
  \label{t:bi}
  $~$
  On $[\bar q_e,+\infty)$
  $$i_e^- \le i_e \le i_e^+,$$
  where 
  \begin{equation}
    \label{iunder}
    i_e^+(q) = \inf_{\pi\in \ms P_e}\left( 
    q \log \frac{q}{R_4(\pi) } - q +R_2(\pi)\right),
  \end{equation}
  and, setting
    $\hat q_e = \sup_{\pi \in \ms P_e} R_4(\pi) > \bar q_e$,
    \begin{equation}
    \label{iover}
    i_e^-(q) = \displaystyle{\begin{cases}
      0 & q\in [\bar q_e,\hat q_e] \\
      \displaystyle{q \log \frac q{\hat q_e} -q + \hat q_e}& q\in (\hat q_e,+\infty).
      \end{cases}}
  \end{equation}
\end{theorem}
In the proof, we will show that the function $i_e^+$
corresponds to the optimal static strategy for the variational problem
\eqref{rfit}, obtained by restricting to paths not depending on time.
It is possible that neither the upper nor lower bound is sharp, which would correspond to the minimiser of
variational problem \eqref{rfit} exhibiting a non-trivial time
dependence. We refer to \cite{BDGJL,BD,GR} for examples of
this phenomenon in other contexts.

The last topic that we discuss is the development by $\Gamma$-converge of the sequence of
functions $\{\ms I_{e,T}(\cdot|m)\}_{T>0}$. In the terminology of large
deviations this corresponds to second order large deviation
estimates. A similar phenomenon has been analysed, in the context of
the so-called east model, in \cite{BLT,BT}.

In contrast to the functional $i_e$ describing the first order
asymptotics, the functional describing the second order asymptotics 
still depend on the initial condition $m$.
We limit the discussion to the particularly relevant case
in which the initial velocities are sampled from the equilibrium
probability. 
As before, we first introduce the limiting function.
Recalling that $M_\epsilon$, $\epsilon\in(0,+\infty)$ denotes the
Maxwellian with zero average velocity and average energy $\epsilon$,
we let $j_e \colon [0,+\infty) \to [0,+\infty]$ be the lower
semicontinuous function defined by
\begin{equation}
  \label{j=}
  j_e(q) \coloneqq
  \begin{cases}
    H_e \big( M_{\epsilon(q)} \big| M_e \big)
    & \textrm{if $q\in [0,\bar q_e]$,} \\
    +\infty & \textrm{if $q>\bar q_e$.}
  \end{cases}
\end{equation}
where $\epsilon(q) \coloneqq e (q/\bar q_e)^2$. 
By direct computation for $q\in [0,\bar q_e]$ we have
\begin{equation}
  \label{j=eq}
  j_e(q) = \log \left(\frac {\bar q_e}q\right)^d.
\end{equation}

\begin{theorem}
  \label{t:gc2}
  The sequence of functions $\big\{\ms I_{e,T}(\cdot |M_e)
  \big\}_{T>0}$ on $[0,+\infty)$ is equi-coercive and
  $\Gamma$-converges to $j_e$ as $T\to\infty$. Namely,  
  \begin{itemize}
  \item [(i)] for each $\ell>0$ there is a compact
    $K_\ell\subset\subset [0,+\infty)$ such that\hfill\break
    $\big\{ q\in [0,+\infty) \colon \ \ms I_{e,T}(\cdot |m) \le \ell
    \big\} \subset K_\ell$ eventually as $T\to \infty$;
  \item [(ii)] for each $q\in[0,+\infty)$ and each sequence $q_T\to q$
    the inequality\hfill\break
    $\liminf_{T\to\infty} \ms I_{e,T}( q_T|M_e) \ge
    j_e(q)$ holds;
  \item [(iii)] for each $q\in[0,+\infty)$ there exists a sequence $q_T\to q$
    such that\hfill\break
    $\limsup_{T\to\infty} \ms I_{e,T}( q_T|M_e) \le j_e(q)$.  
  \end{itemize}
\end{theorem}

In view of standard property of $\Gamma$-convergence, see e.g.\
\cite{Mar}, this statement together with the large deviations upper
bound with fixed $T$ in \cite{BBBC2} implies that the sequence of real
positive random variables $\{q_{N,T}\}_{T,N}$ satisfies a large
deviations upper bound in the limit in which first $N\to\infty$ and
then $T\to \infty$ with speed $N$ and rate function $j_e(\cdot|m)$.
In contrast to the case of the first order asymptotics described by
Theorem~\ref{t:gc}, here the order of the limiting procedure does
matter. In fact, as the large deviations speed does not depend on $T$,
a large deviation principle in the limit in which first $T\to\infty$ and
then $N\to \infty$ would be meaningless.

\section{Reversibility}
The next statement is the counterpart at the level of the rate
functional of the reversibility of the microscopic dynamics. 
Let $\Upsilon\colon \big(\bb R^d\big)^4\to \big(\bb R^d\big)^4$ be the
involution that exchanges the incoming and outgoing velocities, that is  
\begin{equation}
  \label{ups}
  \Upsilon (v,v_*,v',v'_*) = (v',v_*',v,v_*).
\end{equation}
Recalling the definitions of the functionals $H_e$  and $J_{e,T}$ in \eqref{eq:HH}, \eqref{5}, we have the following identity. 

\begin{proposition}
\label{lemma:rev}
  Fix $T>0$. For each $(\bs\pi,\bs Q)\in \ms S_{e,T}$
  \begin{equation}
    \label{eq:rever}
      H_e(\bs\pi_0|M_e)
      + J_{e,T}(\bs\pi,\bs Q) =
      H_e(\bs\pi_T|M_e)
      +    J_{e,T}(\bs\pi,\bs Q \circ \Upsilon).
  \end{equation} 
    Moreover, if either side in \eqref{eq:rever} is finite, then for all $t\in [0,T]$, $\bs \pi_t$ admits a density $f_t$ enjoying the integrability \begin{equation}\label{eq: chain rule well def} \bs Q\left(\left|\log \frac{f' f_\star'}{f f_\star}\right|\right)<\infty.\end{equation}
    Finally, for any $[r, s]\subset [0,T]$, \eqref{eq:rever} 
    the entropy satisfies the chain rule
    \begin{equation}\label{eq: chain rule for entropy} H_e(f_s\de v|M_e) - H_e(f_r\de v|M_e ) =
\bs Q\left(1_{[r,s]}\log\frac{f'f_\star'}{ff_\star}\right).
    \end{equation}
    In particular, $t\mapsto H_e(\bs \pi_t|M_e)$ is finite and continuous in time.
  \end{proposition}
We understand that identity \eqref{eq:rever} holds in the sense that
if either side is finite, then also the other one is finite and 
equality holds.
  The strategy of the proof will be to show that any $(\bs \pi, \bs Q)$ for which the left-hand side $\mathcal{I}_{e,T}(\bs \pi, \bs Q)$ of \eqref{eq:rever} is finite can be approximated by regular trajectories, for which we can compute a version of \eqref{eq: chain rule for entropy} $$ H_e(f_sdv|M_e)-H_e(f_rdv|M_e)=2\int_r^s(Q_u^{(1)}(\log f)-Q_u^{(3)}(\log f)) du $$ where
$\bs Q = \de t \, Q_t$, 
for any $[r,s] \subset [0,T]$, and where the superscript on the right hand side denotes the marginals,
namely $$Q_t^{(1)}(\cdot) = \int Q_t(\,\cdot\,, \de v_*, \de v', \de v_*'), \qquad Q_t^{(3)}(\cdot) = \int Q_t(\de v', \de v_*', \,\cdot\,, \de v_*). $$ On such regular paths, the chain rule can be rearranged to find the identity \eqref{eq:rever}, and the approximations are constructed so as to be able to pass to the limit. The finiteness of the right-hand side of \eqref{eq:rever} will then imply the claimed integrability \eqref{eq: chain rule well def}, which allows us to pass to the limit to find \eqref{eq: chain rule for entropy}.

\begin{proof}
  We fix, for the duration of the proof, a pair $(\bs \pi, \bs Q)$ for which the left hand side of \eqref{eq:rever} is finite; in the following steps, we will show that the right hand side is finite, and
  \begin{equation}\label{eq:chr1}
  H_e(\bs\pi_T|M_e)
  +    J_{e,T}(\bs\pi,\bs Q \circ \Upsilon)
  \le 
  H_e(\bs\pi_0|M_e)
  + J_{e,T}(\bs\pi,\bs Q).
  \end{equation}
  The converse inequality is then proven applying the previous case to the time-reversed path \begin{equation}\label{eq: time reversed path} \hat{\bs \pi}_t:=\bs \pi_{T-t}; \qquad \hat{\bs Q}=(\Upsilon\circ Q_{T-t})dt. \end{equation}


  \subsubsection*{Step 1. Bounds}
  We begin with some estimates. In \cite{BBBC2,He}, it is shown that $J_{e,T}$ admits the variational representation
  $$J_{e,T}(\bs\pi,\bs Q) =
  \sup_{F} (\bs Q(F) - \bs Q^\pi (\ee^F -1 ))=: E(\bs Q\vert \bs Q^\pi)$$
  where the supremum 
  is carried out  over all continuous
  and bounded $F\colon [0,T]\times (\bb R^d)^2\times  (\bb R^d)^2 \to\bb R$
  such that $F(t;v,v_*,v',v'_*)=F(t;v_*,v,v',v'_*)=F(t;v,v_*,v'_*,v')$.
 From this, it immediately follows that $E(\cdot\vert\cdot)$ is convex and
  lower semicontinuous in both arguments.
  
  Since  $J_{e,T}(\bs \pi, \bs Q)$ is finite, by choosing $F = \log(1+|v|+|v_*|)$ 
  and using a standard truncation argument, 
  we obtain that
  $\bs Q(\log(1+|v|+|v_*|))<+\infty$,
  and similarly   $\bs Q(\log(1+|v'|+|v_*'|))<+\infty$.
  Moreover, by choosing $F = -\log(B)$,
  we obtain $\bs Q(\log 1/B) < +\infty$.
  Since $B=|\omega \cdot (v-v_*)|/2$ we conclude that
  $\bs Q(|\log B|)$ is finite.

  Finally, 
  under the additional assumption that $\pi_t$ admits a strictly positive density $f_t>0$ for all $t\in [0,T]$, taking
  $$F= \log\left(\frac 1{(1+|v|)^\alpha (1+|v_*|)^\alpha f(v) f(v_*)}\right)$$
  with $\alpha > d+1$, and recalling that
  $\bs Q(\log(1+|v|+|v_*|)$ is bounded,
  we conclude that $-\bs Q(\log f) <+\infty$.
  Moreover, if $f$ is bounded, then
  $\bs Q(|\log f(v)|)<+\infty$.

  \subsubsection*{Step 2. Velocity regularisation} We first perform a regularisation in the velocity variables.

  Given a pair $(\bs \pi, \bs Q)$,
  we denote by $\bs \pi_t(\de v)=f_t \de v$ and
  $\de \bs Q = Q_t\de t \de v \de v_* \de \omega$.
  Given $0<\ve<1$, let $g_\ve$ be the Gaussian kernel on
  $\bb R^d$ with variance $\ve$, and let $\alpha_\epsilon\to 1$ be given by 
  $\alpha_\epsilon:=\sqrt{\frac d{2e} \ve+1}$.
  For each $\epsilon\in (0,1)$, we construct the new path $(\bs \pi^\ve,\bs Q^\ve)$ 
  by taking the density of $\bs \pi^\ve$ to be 
  $$f^\ve(v) =\alpha_\eps^d (g_\ve * f)(\alpha_\eps v)$$ and setting
  $$\bs Q^\ve =
  \alpha^{2d}_\eps ((g_\ve \otimes g_\ve \otimes \opid) * \bs Q)
  (\alpha_\eps v,\alpha_\eps v^*,\omega).$$ As an immediate result of the definition, it holds that 
  $$\bs \pi_t^\ve (\zeta_0) = \frac 1{\alpha_\eps^2}
  \Big(\frac d2 \ve + \bs \pi_t(\zeta_0)
  \Big) \le e$$ and that $(\bs \pi^\ve,\bs Q^\ve)$ satisfies the balance equation. Using the bounds in Step 1, we can rewrite
  \begin{equation}
    \label{eq:decompJ}
    J_{e,T}(\bs\pi,\bs Q) = E(\bs Q\vert \bs P^{\pi}) - \bs Q(\log B) +
    \bs Q^\pi (1) -
    \bs P^\pi(1)
  \end{equation}
  where $\de {\bs P}^{\pi}=\de t \de \omega \pi_t(\de v)  \pi_t(\de v_*)$.
  Since $\big(\bs P^{\pi}\big)^\ve=\bs P^{\pi^\ve}$, by convexity and
  lower semicontinuity,
 $E(\bs Q^\ve\vert \bs P^{\pi^\ve})\to E(\bs Q\vert \bs P^{\pi})$. Moreover, one can 
  prove
 that
 $\bs Q^\ve (|\log B|)$ is finite
 and 
 \begin{equation}
   \label{eq:convlogB}
   \lim_{\ve \to 0} \bs Q^\ve (\log B) = \bs Q(\log B), \qquad   \lim_{\ve \to 0}  \bs Q^{\pi^\ve}(1) = \bs Q^\pi(1)
 \end{equation}
 Splitting $\log B$ into its positive and negative parts, the convergence of the negative part is guaranteed by the argument of \cite[Appendix A]{BBBC2}. In the positive part, the argument is different, because we no longer (in contrast to the cited paper) assume that $\bs Q(\zeta_0)<\infty$; however, the argument may be completed by noticing that $(\log B)^+ \le \log (1+|v|+|v_\star|)$, which is guaranteed to be integrable thanks to the bounds established in Step 1.

  Moreover, ${\bs P}^{\pi^\ve}(1) =  {\bs P}(1)$.
  Therefore $ J_{e,T}(\bs\pi^\ve,\bs Q^\ve)$ is finite and converges to
  $ J_{e,T}(\bs\pi,\bs Q)$ as $\ve \to 0$.

  \subsubsection*{Step 3. Time regularisation}

 In order to complete the approximation by smooth trajectories, we must also regularise in the time variable; it will be convenient to keep the parameters independent. Writing
  $(f^\ve_t,Q^\ve_t)$ for the  velocity-regularised pair constructed in the
  previous step, let $\imath_\eta$ be the a smooth approximation of the
Dirac measure in $\bb R$,  with support in $(0,\eta)$, and set
  $$(\tilde f^\ve_t,\tilde Q^\ve_t) =
  \begin{cases}
    (f^\ve_r,0) & t\le 0 \\
    (f^\ve_t,Q^\ve_t) & t> 0.
  \end{cases} 
  $$

  We now define $(\bs \pi^{\eta,\ve},\bs Q^{\eta,\ve})$ to be the path
  with densities $f_t^{\eta,\ve} = (\imath_\eta * \tilde f^\ve)_t$ and 
  $Q_t^{\eta,\ve} = (\imath_\eta *  \tilde Q^\ve)_t $, where $\imath_\eta *$
  is  the  convolution in time.
  Observe that $(\bs \pi^{\eta,\ve},\bs Q^{\eta,\ve})$ satisfies the balance equation.

  Let $J_{[0,s]}$ be the functional $J_{e,T}$
  when the interval $[0,T]$ is replaced by $[0,s]$ and we have dropped the
  dependence on $e$. From now on we can follow \cite[Theorem 5.6, Step 3]{BBBC2},
  obtaining that for each $\ve >0$
  $$\lim_{\eta\to 0 }J_{[0,s]}(\bs \pi^{\eta,\ve},
  \bs Q^{\eta,\ve}) = J_{[0,s]}(\bs \pi^\ve,\bs Q^{\ve}).$$

\subsubsection*{Step 4. Entropy chain rule for regular paths.}
  
  By construction, $f^{\eta,\ve}$ is regular in $(t,v)$ and strictly  positive. 
  The same holds for all the marginal densities  $Q_t^{\eta,\ve,(i)}$,
  $i=1,3$, 
  defined as
  $$
  \begin{aligned}
    &Q_t^{\eta,\ve,(1)}(v) = \int \de v_* \de \omega \, Q_t(v,v_*,\omega)\\
    &Q_t^{\eta,\ve,(3)}(v) = \int \de v_* \de \omega \, Q_t(v',v'_*,\omega)
  \end{aligned}
  $$
  where we recall the notation $Q_t^{\eta,\ve,(2)} = Q_t^{\eta,\ve,(1)}$
  and  $Q_t^{\eta,\ve,(4)} = Q_t^{\eta,\ve,(3)}$.
  As a consequence, we can write the
  balance equation pointwise as
  $$\partial_t f_t^{\eta,\ve}  = 2 (Q_t^{\eta,\ve,(3)} - Q_t^{\eta,\ve,(1)}).$$
    Since $\bs \pi^{\eta, \ve}$ admits bounded densities $f^{\eta,\ve}(t,v)$,
    the final bound in Step 1 applies to prove
  \begin{equation}\label{eq: integrability of f}
  \bs Q^{\eta,\ve}(|\log f^{\eta,\ve}(v)|+ |\log f^{\eta,\ve}(v')|)
= \bs Q^{\eta,\ve,(1)} \big(|\log f^{\eta,\ve }|\big)+ \bs Q^{\eta,\ve,(3)}
\big(|\log f^{\eta,\ve}|\big)<+\infty.
\end{equation}
Therefore, using that
$E(\bs Q^{\eta,\ve}\vert \bs P^{\bs \pi^{\eta,\ve}})<+\infty$,
as follows from \eqref{eq:decompJ} and $\bs Q^{\eta,\ve}(|\log B|)<+\infty$,
we get that  
$\bs Q^{\eta,\ve} \big(|\log Q^{\eta,\ve}| \big)$ is finite. For any $t\in [0,T]$,
\begin{equation}
  \label{eq:dtflogf}
  \partial_t (f_t^{\eta,\ve} \log f_t^{\eta,\ve})  =
  2 Q_t^{\eta,\ve,(3)} (1+\log f_t^{\eta,\ve})
  - 2 Q_t^{\eta,\ve,(1)} (1+\log f_t^{\eta,\ve}).
\end{equation}
Since the reference measure is the Maxwellian $M_e$, we may use the representation \eqref{eq:HH} of $H_e(\cdot|M_e)$, so that the problem reduces to studying the evolution of $\int f^{\eta, \epsilon}_t \log f^{\eta, \epsilon}_t dv$. At time $t=0$,
$H_{e}(f_0^{\eta,\eps}\de v|M_e) = H_{e}(f_0^{0,\eps}\de v|M_e)$
is finite.
Integrating in $t\in [0,s]$ and in $v\in \bb R^d$,
we conclude that
\begin{equation}\label{eq: step 4 conclusion} H_e(f^{\eta,\ve}_s\de v|M_e) - H_e(f^{\ve}_0\de v|M_e) =
2 \bs Q_{[0,s]}^{\eta,\ve} (\log {f^{\eta,\ve}}(v')) -   
2 \bs Q_{[0,s]}^{\eta,\ve} (\log {f^{\eta,\ve}}(v))
\end{equation}
where $\bs Q_{[0,s]}$ is the restriction of $\bs Q$ to the time window $[0,s]$.
\subsubsection*{Step 5. Computation on regularised paths}
We now perform a computation, still at the level of the regularised paths, in order to link the integrals appearing on the right-hand side of \eqref{eq: step 4 conclusion} to the rate function $J_{[0,s]}$ appearing in the desired conclusion \eqref{eq:rever}; we refer to \cite[Section 6.5.2]{He2} for a similar argument. Since $f^{\eta, \ve}$ is everywhere positive, one can check that $\mathbf{Q}^{\pi^{\eta, \ve}}\circ \Upsilon$ is absolutely continuous with respect to $\mathbf{Q}^{\pi^{\eta, \ve}}$ with a density given by $$ \frac{\de
  (\mathbf{Q}^{\pi^{\eta, \ve}}\circ \Upsilon)}
{\de \mathbf{Q}^{\pi^{\eta, \ve}}}=
\frac{f^{\eta, \ve}_t(v')f^{\eta, \ve}_t(v_\star')}
{f^{\eta, \ve}_t(v)f^{\eta, \ve}_t(v_\star)}.$$
Meanwhile, since $\Upsilon$ is an involution and the finiteness of $J_{e,T}(\bs \pi^{\eta,\ve}, Q^{\eta, \ve})$ implies that $\bs Q^{\eta, \ve}$ is absolutely continuous with respect to $\bs Q^{\pi^{\eta, \ve}}$, it follows that $\mathbf{Q}^{\eta, \ve}\circ \Upsilon$ has a density with respect to $\mathbf{Q}^{\pi^{\eta, \ve}}\circ \Upsilon$ given by $\frac{\de \mathbf{Q}^{\eta, \ve}}{\de \mathbf{Q}^{\pi^{\eta, \ve}}}\circ\Upsilon$. By the chain rule for Radon-Nidokym derivatives, we finally see that $\mathbf{Q}^{\eta, \ve}\circ \Upsilon$ is absolutely continuous with respect to $\bs Q^{\pi^{\eta,\ve}}$ and
$$\frac{\de (\mathbf{Q}^{\eta, \ve}\circ \Upsilon)}
{\de \mathbf{Q}^{\pi^{\eta, \ve}}}= \frac{f^{\eta, \ve}_t(v')f^{\eta, \ve}_t(v_\star')}{f^{\eta, \ve}_t(v)f^{\eta, \ve}_t(v_\star)}
\left(\frac{\de \mathbf{Q}^{\eta, \ve}}{\de \mathbf{Q}^{\pi^{\eta, \ve}}}
  \circ \Upsilon\right).$$
Substituting into the definition \eqref{5} of $J$ and using the additivity of the logarithm, it follows that \begin{equation}
	\label{eq: balance1} J_{[0,s]}(\bs \pi^{\eta,\ve},\bs Q^{\eta,\ve}\circ \Upsilon) =  J_{[0,s]}(\bs \pi^{\eta,\ve},\bs Q^{\eta,\ve})+ (\mathbf{Q}_{[0,s]}^{\eta, \ve}\circ \Upsilon)\left(\log \frac{f^{\eta, \ve}(v')f^{\eta, \ve}(v_\star')}{f^{\eta, \ve}(v)f^{\eta, \ve}(v_\star)}\right).
\end{equation} Thanks to \eqref{eq: integrability of f}, the definition of $\Upsilon$ and the symmetry in $(v, v')\leftrightarrow (v_\star, v_\star')$, we may further break up the last term as \begin{equation}\begin{split}
	(\mathbf{Q}_{[0,s]}^{\eta, \ve}\circ \Upsilon)\left(\log \frac{f^{\eta, \ve}(v')f^{\eta, \ve}(v_\star')}{f^{\eta, \ve}(v)f^{\eta, \ve}(v_\star)}\right)&= 2\mathbf{Q}_{[0,s]}^{\eta, \ve}(\log f^{\eta, \ve}(v))- 2\mathbf{Q}_{[0,s]}^{\eta, \ve}(\log f^{\eta, \ve}(v')).  \end{split}
\end{equation} Together with \eqref{eq: step 4 conclusion}, we obtain the balance equation for the entropy on the regularised paths
\begin{equation}
  \label{eq:chreg}
  H_e(f^{\eta,\ve}_s\de v| M_e)
+ J_{[0,s]}(\bs \pi^{\eta,\ve},\bs Q^{\eta,\ve}\circ \Upsilon)
= H_e(f^{\ve}_0 \de v| M_e) +
J_{[0,s]}(\bs \pi^{\eta,\ve},\bs Q^{\eta,\ve}).
\end{equation}

\subsubsection*{Step 6. Proof of \eqref{eq:rever}}

We now pass to the limit $\eta\to 0$ in \eqref{eq:chreg}.
The right hand side converges
by Step 3.
By lower semicontinuity 
$$H_e(f^{\ve}_s\de v| M_e)
+ J_{[0,s]}(\bs \pi^{\ve},\bs Q^{\ve}\circ \Upsilon)
\le  H_e(f^{\ve}_0 \de v| M_e) +
J_{[0,s]}(\bs \pi^{\ve},\bs Q^{\ve}).
$$
Now we pass to the limit $\ve\to 0$.
By \eqref{eq:HH}, $H_e(\cdot |M_e)$ is convex.
Then, by Jensen's inequality, Step 2, and the lower semicontinuity,
we deduce that the more general version of \eqref{eq:chr1} for any time interval $[0,s]$: \begin{equation}\label{eq:chr1}
  H_e(\pi_s|M_e)
  +    J_{[0,s]}(\bs\pi,\bs Q \circ \Upsilon)
  \le 
  H_e(\pi_0|M_e)
  + J_{[0,s]}(\bs\pi,\bs Q).
\end{equation} This extends to any interval $[r,s]\subset [0,T]$ by taking differences, and the special case $s=T$ yields \eqref{eq:chr1}. The same argument applied to the time-reversed path \eqref{eq: time reversed path}
shows that the previous inequality is actually an equality, and so is the version applied to a sub-interval $[r,s]$: \begin{equation}\label{eq:chr2}
  H_e(\pi_s|M_e)
  +    J_{[r,s]}(\bs\pi,\bs Q \circ \Upsilon)
  = 
  H_e(\pi_r|M_e)
  + J_{[r,s]}(\bs\pi,\bs Q).
  \end{equation} generalising the the claim \eqref{eq:rever}. 
  
  \subsubsection*{Step 7. Density, integrability and chain rule.}

  Let us now specialise to the case where the left-hand side, and
  hence both sides, of \eqref{eq:rever} are finite.  As a result, the
  right-hand side of \eqref{eq:chr1} is bounded in $s\in [0,T]$, and
  since $J$ is nonnegative, it follows that
  $\sup_{s\le T} H_e(\pi_s|M_e)<\infty$, from which it follows that
  every $\pi_s$ admits a density $\pi_s=f_sdv$. Using the finiteness
  of $J_{[0,T]}(\bs \pi, \bs Q\circ \Upsilon)$, the same argument as
  in Step 5 identifies

  \begin{equation}
    \frac{\de (\bs Q\circ \Upsilon)}{\de \bs Q^\pi}= \frac{f_t(v') f_t(v_\star')}{f_t(v)f_t(v_\star)}\left(\frac{\de \bs Q}{\de \bs Q^{\pi}}\circ\Upsilon\right)
  \end{equation}
  and, in particular, the first factor is finite $\bs Q^\pi$-almost everywhere. As a result, it follows that
  \begin{equation*}
    \left(\log \frac{f'f_\star'}{ff_\star}\right)_+ \le \left(\log \frac{\de (\bs Q\circ \Upsilon)}{\de \bs Q^\pi}\right)_++
    \left(\log \frac{\de \bs Q}{\de \bs Q^\pi}\circ \Upsilon\right)_-
  \end{equation*}
  with $\pm$ denoting the positive, respectively negative, parts of a function, and $ _\star, '$ indicating the arguments at which $f_t$ is evaluated. We now integrate both sides with respect to $\bs Q\circ \Upsilon$ and use the finiteness of
  $$ J_{[0,T]}(\bs \pi, \bs Q\circ \Upsilon)=
  (\bs Q\circ \Upsilon)\left(\log \frac{\de
      (\bs Q\circ \Upsilon)}{\de \bs Q^\pi}\right)
  -(\bs Q\circ \Upsilon)(1)+\bs Q^\pi(1) $$
  to get 
\begin{equation*}
 \begin{split} 
   \bs Q\left(\left(\log \frac{f f_\star}{f' f_\star'}\right)_+\right)
   &\le  J_{[0,T]}(\bs \pi, \bs Q\circ \Upsilon)+ (\bs Q\circ \Upsilon)(1)\\
   &+
   (\bs Q\circ \Upsilon)
   \left(\left(\log \frac{\de
      (\bs Q\circ \Upsilon)}{\de \bs Q^\pi}\right)_-\right)
   + \bs Q \left(\left(\log\frac{\de \bs Q}{\de \bs Q^\pi}\right)_-\right).
    \end{split}
     \end{equation*}
  The last two terms are readily seen to be finite, using the
  boundedness of $k(\log k)_-$ and the finiteness of $\bs Q^\pi(1)$,
  and we ultimately conclude that
$$ 
\bs Q\left(\left(\log \frac{ff_\star}{f'f'_\star}\right)_+\right)<\infty.
$$ 
The argument for the negative part is similar, and we conclude the claimed
integrability  \eqref{eq: chain rule well def}.

The deduction of \eqref{eq: chain rule for entropy} from \eqref{eq:chr2} proceeds as in Step 5, and the continuity of $t\mapsto H_e(\bs \pi_t|M_e)$ follows by dominated convergence. 
\end{proof}

\section{Controllability of the Boltzmann equation}
\label{s:3}
The following result, which will be used in deriving Proposition \ref{prop:ie} and Theorem \ref{t:gc}, shows that any two probability measures may be joined by a path of finite cost.
  \begin{theorem}
  \label{t:cbe}
  Fix $T>0$ and $e>0$. There exists $C_1=C_1(e)$ and a function $F_1:[0,\infty)^3\to [0,\infty)$ for which the following holds.
  Given $\pi_i \in \ms P_e$ with bounded entropy, $i=1,2$ and $\kappa\ge C_1(e)$, there exists a path $(\bs \pi,\bs Q)\in \mc S_{e,T}$ such
  that $\bs\pi_0=\pi_1$, $\bs \pi_T=\pi_2$, and
  \begin{equation*}
    J_{e,T}(\bs\pi,\bs Q) \le F_1(e,T,\kappa) + H_e(\pi_2|M_e),
    \qquad  \bs Q(1) = \kappa.
  \end{equation*} The function $F$ may be chosen such that, for fixed $e, \kappa$, for all $0<T_0<T_1<\infty$, it holds that \begin{equation}\sup_{T_0\le T\le T_1} F_1(e,T,\kappa)<\infty \label{eq: boundedness of F} \end{equation}
  \end{theorem}


 In order to contextualise this result, which may have
independent interest, we first discuss an alternative
formulation of the dynamical rate function $J_{e,T}$ in \eqref{5}
in terms of a control problem for the homogeneous Boltzmann equation
\eqref{hbe}, see also \cite{Le,He}.

  On a fixed time interval $[0,T]$, we define a cost functional \begin{equation} \label{alt J}
  \widehat J_{e,T} ({\bs \pi}) :=
  \frac 12 \inf\left\{\int_0^T\!\de t \int \bs \pi_t^{}(dv)\bs \pi_t^{}(dv_*)
  \de \omega\, B(v-v_*,\omega) \Psi({\bs F})\right\}.
\end{equation}
where $\Psi\colon \bb R \to \bb R_+$ is defined by
$\Psi(x) := xe^{x}-e^x+1$, and the infimum runs over all controls  $\bs F\colon [0,T]\times
\big(\bb R^d\big)^4\to \bb R$ for which the density $f_t$ of $\pi_t$ is a weak solution to the \emph{controlled Boltzmann equation}
\begin{equation}
  \label{ccp}
  \begin{cases}
    \displaystyle{ \partial_t f_t(v) = \int\!\de v_*\de \omega \,
      B(v-v_*, \omega)  e^{\bs F_t(v',v'_*,v,v_*)}  f_t(v') f_t(v_*') }  \\
    \displaystyle{
      \phantom{\partial_t f_t(v) =}
      - \int\!\de v_*\de \omega \, B(v-v_*, \omega)
       e^{\bs F_t(v,v_*,v',v_*')} f_t(v) f_t(v_*)}, 
      \quad (t,v)\in (0,T)\times\bb R^d
      \\
      f_0\de v = \pi_0
  \end{cases}
\end{equation} with the energy never exceeding $e$.
Note that $\Psi$ is positive with a unique
quadratic minimum achieved at $x=0$. Observe also that the case
${\bs F}=-\infty$, for which $ \widehat J_{e,T} ({\bs F})< +\infty$,
corresponds to a  vanishing perturbed collision kernel.
With this definition, it holds for all $\bs \pi$ that \begin{equation}\label{eq: control form of rate} \widehat{J}_{e,T}(\bs \pi)= \inf_{\bs Q} J_{e,T}(\bs \pi, \bs Q)\end{equation}
In this way, Theorem \ref{t:cbe} may be understood as asserting that, given $\pi_1,\pi_2\in \ms P_e$ with bounded entropy,
there exists a control $\bs F$ and a solution $\bs \pi$ to \eqref{ccp} such that $\bs \pi_0=\pi_1$,
$\bs\pi_T=\pi_2$, and such that the integral appearing in \eqref{alt J} is finite.

Let us remark that 
arguments from \eqref{alt J} run into issues of nonuniqueness, even for the (uncontrolled) Boltzmann equation \eqref{hbe}. Given an initial datum $f_0$, taking $F=0$ reduces \eqref{ccp} to the uncontrolled Boltzmann equation \eqref{hbe}, for which there are multiple solutions \cite{LuW}.
For this reason, we will not make precise the notion of
admissible controls, and have rather formulated Theorem \ref{t:cbe} in terms of $J_{e,T}$. We emphasise that we do not need the initial and target measures to have the same energy.


  \subsubsection*{Strategy of the Proof} Before proving the full statement of Theorem \ref{t:cbe}, we first prove a `one-sided' version in Lemma \ref{t:ct}, which is the special case where $\pi_2$ is a Maxwellian
  $M_e$. In this case, we can construct the path $\bs \pi$ by taking a solution to the homogeneous Boltzmann equation, starting at $\pi_1$, and which has energy $e$ for all times $t>0$, and then reparametrising time. Such a solution always exists, either as the unique energy-conserving solution if $\pi_1$ already has energy $e$, or as a Lu-Wennberg solution otherwise. The total collision rate will be tuned to a desired value $\kappa$ by using two different flux measures $Q^1_\tau, Q^2_\tau$ for the non-reparametrised solution: $Q^1_\tau$ will be the `natural' measure associated to the solution, while $Q^2_\tau$ will remove certain collisions. By switching over after a suitably chosen time, we can guarantee that the total number of collisions reaches a desired, finite value $\kappa$, which will also ensure that the rate remains finite after the change of time-scale. The `two-sided' case is argued separately, and is given by concatenating the path $\pi_1\to M_e$ produced by Lemma \ref{t:ct} with the time-reversal of the path $\pi_2\to M_e$, using Proposition \ref{lemma:rev}.

\begin{lemma}
  \label{t:ct}
For every $e, T>0$, there exist $C_2=C_2(e)$ and a function $F_2:[0,\infty)^3\to [0,\infty)$ satisfying \eqref{eq: boundedness of F} such that, for every $\pi \in \ms P_e$ and $\kappa\ge C_1(e)$, there exists $(\bs\pi,\bs Q)\in \mc S_{e,T}$ such that
  $\bs{\pi}_0=\pi$, $\bs\pi_T=M_e$, and
  \begin{equation*}
        J_{e,T}(\bs \pi,\bs Q) \le F_2(e, T, \kappa), \qquad  \bs Q(1) =\kappa.
  \end{equation*} 
\end{lemma}

\begin{proof} We divide into steps; fix, everywhere, $e, T$ and $\pi$ as in the lemma.  A non-reparametrised solution is defined in step 1, and the reparametrisation is given in step 2, yielding the final $(\bs \pi, \bs Q)$. The asserted bounds are proven in steps 3-5.  
  \subsubsection*{Step 1. Infinite-time path.}
  Let us take $\big(f_\tau\big)_{\tau\ge 0}$ to be \emph{any} solution to the homogeneous Boltzmann equation whose initial datum $f_0$ is the density of $\pi$ and whose associated measures $\pi_\tau=f_\tau dv$ satisfy $\pi_\tau(\zeta_0)=e$ for all $\tau>0$. In the case when $\pi$ already has energy $e$, then $f_\tau$ is the unique energy conserving solution, while if $\pi(\zeta_0)<e$, then it is a Lu-Wennberg solution with a single jump of the energy at $\tau=0$. In view of \cite[Lemma~7.6]{BBBC2}, such a solution exists and setting $${Q}^1_\tau\coloneqq Q^{f_\tau} (\de v,\de v_*,\de \omega) \coloneqq \frac 12
  f_\tau(v)f_\tau(v_*) B(v-v_*,\omega) \de v \de v_* \de \omega$$ produces paths in $\mathcal{S}_{e,T}$ with vanishing dynamical cost for any finite time horizon $T$. Moreover, since $\pi_t$ is energy conserving away from $t=0$, it also satisfies any bounds valid for the energy conserving equation (e.g. \cite{Mou}) which only require finite energy, replacing the initial energy by $e$ if necessary. \\ \\ In the sequel, we will use a different flow measure associated to the solution $(f_\tau)_{\tau\ge 0}$ in addition to $Q^1_t$ given above. \begin{equation}
    \label{q:=}
    Q^2_{\tau}(\de v ,\de v_*,\de \omega)
    =  \big[ f_\tau(v) f_\tau(v_*)
    - f_\tau(v') f_\tau(v_*') \big]_+
    B(v-v_*,\omega) \de v \de v_*  \de \omega,
  \end{equation}
  in which $[\,\cdot\,]_+$ denotes the positive part.
  \subsubsection*{Step 2.  Time  reparametrisation} We now construct a path with a specified number of collisions, on a finite time interval $[0,T]$. First, we observe that, for any $\tau>0$, $Q^1_\tau(1)\ge Q^2_\tau(1)$. Moreover, a straightforward argument shows that $Q^1_\tau(1)$ is bounded below, uniformly in $\tau, \pi$, so $\int_0^\infty Q^1_\tau(1)d\tau=\infty$.  Let us define $\kappa_\star(\pi):=\int_0^\infty Q^2_\tau(1) d\tau$; we will see in Step 3 below that there exists $C_2(e)<\infty$ such that $\kappa_\star(\pi)\le C_2(e)$ for any $\pi \in \ms P_e.$ For any $\kappa\in[ C_2(e),\infty)$, there is thus a unique $\tau_\star \in [0,\infty)$ such that $$ \int_0^{\tau_\star} (Q^1_\sigma-Q^2_\sigma)(1)d\sigma= \kappa-\kappa_\star(\pi) $$ and set \begin{equation}
 	Q_\tau:=\begin{cases} Q^1_\tau, & \text{if } 0\le \tau<\tau_\star; \\ Q^2_\tau, & \text{otherwise} \end{cases}
 \end{equation} to find that $\int_0^\infty Q_\tau(1)d\tau = \kappa$. We note, for future reference, that there exists $c(e)<\infty$ such that $Q^1_\tau(1)\ge c(e)^{-1}$ for all $\tau$, from which it follows that \begin{equation}\label{eq: bound taustar} \tau_\star \le c(e)\kappa. \end{equation}  Setting $\phi\colon [0,T) \to [0,+\infty)$ to be
  $\phi(t)=(T-t)^{-1} -T^{-1}$, we define a time-reparametrised path $(\bs\pi,\bs Q)$ by
  \begin{equation}
    \label{qrisc} 
    \bs\pi_t = f_{\phi(t)} \de v,\qquad
    \bs Q(\de t) = Q_{\phi(t)}\phi'(t) \de t
  \end{equation}
  where we understand that $\bs \pi_T=M_e$. 
In the remaining steps, we establish an upper bound $\kappa_\star(\pi)\le C_1(e)$, which allows the previous construction for any $\kappa\ge C_1(e)$, that $(\bs \pi, \bs Q)\in \mathcal{S}_{e,T}$, and the claimed bound on $J_{e,T}(\bs \pi, \bs Q)$.

  \subsubsection*{Step 3. Bound on $\kappa_\star(\pi)$} In this step, we prove that there exists $C_2=C_2(e)$, depending only on the energy $e$, such that $\kappa_\star(\pi)\le C_2(e)$ for all $\pi\in \ms P_e$. We denote by $\big\| \cdot \big\|_{\mathrm{TV}}$ the total
  variation norm on the space of signed measures.
  By \cite{Mou}, there exist $\gamma=\gamma(e)>0$ and $C=C(e)<+\infty$ such
  that for any $f_0$ with energy $e$
  \begin{equation}
    \label{cexp}
    \big\| f_\tau \de v -M_e \big\|_{\mathrm{TV}} \le  C e^{-\gamma \tau},
    \qquad \tau\ge 0.
  \end{equation}
Writing $g$ for the density of $M_e$, recalling \eqref{q:=},
  and using that $g(v)g(v_*) B(v-v_*,\omega)$ is symmetric with
  respect to $\Upsilon$, we have 
  \begin{equation*}
    \begin{split}
    & \big\| Q^{f_\tau} -Q^{M_e} \big\|_{\mathrm{TV}}
    \le
    \int \!\de v \de v_*\de \omega
    \,
    \big| f_\tau(v)f_\tau(v_*) - g(v) g(v_*)\big| B(v-v_*,\omega)
    \\
    &\quad
    \le
    2 \int \!\de v  \de v_*\, |v|  \,
    \big| f_\tau(v)f_\tau(v_*) - g(v) g(v_*)\big| 
    \\
    &\quad
    \le
    2 \big\| f_\tau \de v- M_e \big\|_{\mathrm{TV}}
    \, \int \!\de v \, |v| f_\tau(v)
    + 2  \int \!\de v  \, |v| \big| f_\tau(v) -g(v) \big|.
  \end{split}
  \end{equation*}
  Since $f_\tau$ has energy $e$ for any $\tau>0$, the first term can be directly
  bounded by using \eqref{cexp}.  In order to bound the second, given
  $\ell >0$ we have
  \begin{equation*}
    \begin{split}
    &\int \!\de v \, |v| \big| f_\tau(v) -g(v) \big|
    \\
    &\quad \le \ell  \int \!\de v \, \big| f_\tau(v) -g(v) \big|
    + \int_{|v|\ge \ell} \!\de v \, |v|\, \big| f_\tau(v) -g(v) \big|
        \le
    \ell C e^{-\gamma \tau} + 4\, \frac {e}{\ell}
  \end{split}
\end{equation*}
where we used \eqref{cexp} and Chebyshev inequality noticing that both
$ f_\tau$ and $g$ have energy $e$.
Optimising at $\ell= e^{\gamma \tau/2}$, we find that, for some $\gamma'=\gamma'(e)>0$ and
  $C'=C'(e)<+\infty$,
   \begin{equation}
    \label{cexp'}
    \big\| Q^{f_\tau} -Q^{M_e} \big\|_{\mathrm{TV}} \le C' e^{-\gamma'\tau},
    \qquad \tau\ge 0.
  \end{equation}

Recalling the definition \eqref{q:=} of $Q^2_\tau$, the bound
\eqref{cexp'} implies
   \begin{equation}
    \label{qtau0}
   Q^2_\tau(1)\le  2C' e^{-\gamma' \tau}.
  \end{equation}
  Recalling that $C', \gamma'$ depend only on $e$ and not on $\pi \in \ms P_e$, the claim thus holds when we define $C_2(e):=2C'(e)/\gamma'(e)$, since $\kappa_\star(\pi)$ is defined to be the integral of the left-hand side over $t\in [0,\infty)$.
  
  \subsubsection*{Step 4. Balance equation.}
  We now check that \eqref{bal} holds, first on the interval $[0,T)$, and then extending to the endpoint by continuity. First, by definition of $(f_\tau)_{\tau>0}$ and $Q_\tau$, $\big(
  \big( (f_\tau)_{\tau\ge 0} \de v, (Q_\tau)_{\tau \ge 0} \de \tau \big)$
  satisfies the balance equation \eqref{bal} on $[0, \tau_\star]$, since on this interval $Q_\tau=Q^1_\tau$, and because $f_\tau$ satisfies the Boltzmann equation \eqref{hbe}. For $\tau\ge \tau_\star$, observe, recalling \eqref{ups},
  \begin{equation}
    \label{q=qf}
    Q^{f_{\tau}} =Q_\tau
    +\frac 12 \big( Q^{f_\tau} + Q^{f_\tau}\circ \Upsilon \big)
    - \frac 12 \big| Q^{f_\tau} - Q^{f_\tau}\circ \Upsilon  \big|
  \end{equation}
  Since the last two terms on the right hand side above are
  symmetric with respect to $\Upsilon$, 
  \begin{equation*}
    \frac{\partial f_\tau}{\partial \tau} \de v=
    2 \big[ \big(Q^{f_\tau}\big)^{(3)} (\de v)
    - \big(Q^{f_\tau}\big)^{(1)} (\de v) \big]
    = 2\big[ \big(Q_\tau\big)^{(3)}(\de v) - \big(Q_\tau\big)^{(1)}
    (\de v) \big]
  \end{equation*}
  where the superscripts denote the marginals on the first and third
  variable, respectively, and hence the pair $\big(
  \big(  (f_\tau)_{\tau\ge 0} \de v, (Q_\tau)_{\tau \ge 0} \de \tau \big)$
  also satisfies the balance equation \eqref{bal} on $\tau \ge \tau_\star$, and hence globally.
  Therefore, by change of the time variable, the pair $(\bs\pi, \bs Q)$
  defined in \eqref{qrisc} also satisfies the balance equation on $[0,T)$, in that \eqref{bal} holds when the terminal time $T$ is replaced by any $t\in [0,T)$. In order to extend this to the endpoint $T$, we observe that $$ \int_t^T \bs Q(ds,1) \le 2C'\int_t^T (T-s)^{-2}\exp\left(-\gamma'((T-s)^{-1}-T^{-1})\right)ds \to 0 $$ as $t\uparrow T$. In particular, if we may pass to the limit of both sides when we evaluate the balance equation at $t<T$ and send $t\to T$. It follows that $(\bs \pi, \bs Q)$ satisfies the balance equation on $[0,T]$, and that 
  $(\bs\pi, \bs Q)\in \ms S_{e,T}$.

 \subsubsection*{Step 5. Estimate of dynamic cost.} To estimate $J_{e,T}(\bs\pi,\bs Q)$, we notice that  \eqref{q=qf}
  implies $Q_\tau\le Q^{f_\tau}$, independently of whether or not $\tau\le \tau_\star$. We thus have
  \begin{equation*}
    \begin{split}
    J_{e,T}(\bs \pi,\bs Q)
    &
    = \int_0^T\!\de t \, \phi'(t)
    \int \de Q_{\phi(t)} \log\Big(\phi'(t)
    \frac{ d Q_{\phi(t)}}{ d Q^{f_{\phi(t)}}} \Big)
    -\bs Q(1)
    + \int_0^T\!\de t\, Q^{f_{\phi(t)}}(1)
    \\
    &\le \int_0^\infty\!\de \tau\,  Q_{\tau}(1)\,
    \log \big(T^{-1}+\tau\big)^2 + T
    \sup_{\mu \in \ms P_e} \frac 12
    \int\! \mu (\de v) \mu (\de v_*) \de \omega\,  B.
 \end{split}
\end{equation*}
The final term may be estimated by
$c (2+4e)T^{-2}$ using a simple upper bound $B\le \frac12(1+|v-v_\star|^2)$, which is of the form required for \eqref{eq: boundedness of F}. In the first integral, the contribution from $\tau\le \tau_\star$ is bounded by recalling that $\tau_\star(\pi)\le c(e)\kappa$ and that $Q_\tau(1)\le (2+4e)$, yielding a bound $(4+8e)c(e)\kappa \log (T^{-1}+c(e)\kappa)$, while the contribution from $\tau>\tau_\star$ is bounded by using \eqref{qtau0}. All of these bounds depend only on $e, T, \kappa$ and have the property \eqref{eq: boundedness of F} asserted in the Lemma, so the proof is complete.   
\end{proof}

\begin{proof}[Proof of Theorem~\ref{t:cbe}]
  Fix $e$. Let us define $C_1(e):=2C_2(e)$, where $C_2(e)$ is given by Lemma \ref{t:ct}. For any $\kappa\ge C_1(e)$, let $({\bs \pi}^i, {\bs Q}^i)\in \ms S_{e,T/2}$, $i=1,2$, be the
  paths provided by Lemma~\ref{t:ct} with $\pi=\pi^i$, $i=1,2$ and $T$
  replaced by $T/2$ and $\kappa$ replaced by $\frac\kappa2$. Denote by $\chi\colon [0,T/2]\to [0,T/2]$ be the
  time reflection $\chi(t) =T/2-t$ and let
  $\big({\hat{\bs \pi}}^2, {\hat{\bs Q}}^2\big)$ be the path defined by
  \begin{equation*}
    {\hat{\bs \pi}}^2_t :=  {{\bs \pi}}^2_{\chi(t)},\qquad
     {\hat{\bs Q}}^2 := {\bs Q}^2 \circ \chi \circ \Upsilon.
  \end{equation*}
  By direct computation, it satisfies the balance equation \eqref{bal}
  so that $\big({\hat{\bs \pi}}^2, {\hat{\bs Q}}^2\big)\in \ms
  S_{e,T/2}$.
Finally, let $(\bs \pi,\bs Q)\in \ms S_{e,T}$ be the path obtained by
  concatenating $\big(\bs \pi^1,\bs Q^1\big)$ with
  $\big({\hat{\bs \pi}}^2, {\hat{\bs Q}}^2\big)$.  By construction
  $\bs\pi_0=\pi_1$, $\bs\pi_T=\pi_2$,
  $\bs Q(1) = {\bs Q}^1(1)+{\bs Q}^2(1)=\kappa$, and thanks to
  Proposition~\ref{lemma:rev},
  \begin{equation*} \begin{split}
      J_{e,T}(\bs \pi,\bs Q) & =
      J_{e,T/2}(\bs \pi^1,\bs Q^1)+
      J_{e,T/2}(\bs \pi^2,\bs Q^2) + H_e(\pi_2|M_e) \\ & \le 2F_2(e, T/2, \kappa/2) + H_e(\pi_2|M_e) \end{split}
  \end{equation*}
 where $F_2$ is the function given by Lemma~\ref{t:ct}. This completes the proof with $F_1(e, T, \kappa):=2F_2(e, T/2, \kappa/2)$. 
\end{proof}

\section{First order asymptotics}

We collect the proofs of the main results related to the
first order asymptotic.
\begin{proof}[Proof of Proposition \ref{prop:ie}.]$~$

  \subsubsection*{Step 1. The identity \eqref{eq:step1}} 

  Since all paths $(\bs \pi, \bs Q)\in \ms A_{e,T}(q|\mu)$ start and end at $\pi_0=\pi_T=\mu$, one may concatenate competitors and use the translation
  covariance of $J_{e,T}$ to find that the function \begin{equation} \label{eq: def} i_e(q|\mu, T):=\inf_{(\bf \pi, \bs Q)\in \mathcal{A}_{e,T}(q|\mu)}J_{e,T}(\bs \pi, \bs Q) \end{equation} enjoys the subadditivity property\begin{equation}
  	i_e(q|\mu, T_1+T_2) \le i_e(q|\mu, T_1)+i_e(q|\mu, T_2).
  \end{equation} From this, a standard argument yields that $i_e(q|\mu, T)/T\to i_e(q|\mu)$, which is the content of the assertion \eqref{eq:step1}.

  
  \subsubsection*{Step 2. Lower semicontinuity of $i_e(\cdot |\mu)$}

  Fix $q_0,q>0$ and set $\lambda =  q/q_0$. 
  Given $(\bs \pi,\bs Q)\in \ms A_{e,T}(q|\mu)$, define
  $$\tilde \pi^\lambda_t  = \pi_{t/\lambda}, \ \
  \tilde Q^\lambda (\de t) = \frac 1\lambda Q(\de t/\lambda)$$
  for $\lambda > 0$. In particular $(\tilde \pi^\lambda, \tilde Q^\lambda)
  \in \ms A_{\lambda T}(q_0|\mu)$.
  By direct computation
  $$\begin{aligned}
    J_{e,\lambda T} (\tilde \pi^\lambda, \tilde Q^\lambda) &=
  \int_0^{\lambda T} \tilde Q^\lambda(\de t) \log \frac {\de \tilde Q}
  {\ \ \de Q^{\tilde \pi^\lambda}}
  - \tilde Q^\lambda(1)+  Q^{\tilde \pi^\lambda} (1)
  \\
  &= \lambda T  q_0  \log \frac 1 \lambda
  + ( \lambda -1 ) Q^\pi(1) + J_{e,T}(\bs \pi,\bs Q)
  \end{aligned}
  $$
  Since $Q^\pi(1) \le cT$, we find, for some constant $c$ depending on $e$,
  $$i_e(q_0| \mu) \le
  \frac 1\lambda i_e(q |\mu)
  + \frac q\lambda  \log \lambda + \big| 1-\frac 1\lambda\big| c =
  \frac {q_0}q  i_e(q |\mu)+
   q_0  \log \frac {q_0}q  + \big| 1- \frac {q_0}q\big| c 
  $$
  Taking the limit inferior for $q\to q_0$
  produces the lower-semicontinuity of $i_e$.

  \subsubsection* {Step 3. Independence of $i_e(q |\mu)$ on $\mu$}

  We first prove, for any $\mu\in \ms P_e$, the inequality
  \begin{equation}
    \label{eq:muMe}
    i_e(q|\mu) \le i_e(q|M_e).
  \end{equation}
  Fix $T>0$ and $(\bs \pi,\bs Q) \in \ms A_{e,T}(q|M_e)$, and let $(\hat{\bs \pi}, \hat{\bs Q})
  \in  \ms S_{e,1}$ be the path satisfying
  $\hat \pi_0 = \mu$, $\hat \pi_1 = M_e$, provided by
Lemma \ref{t:ct}.
  Set $\tilde T = T+2$ and $(\tilde {\bs \pi}, \tilde{\bs Q})\in S_{e,\tilde T}$ be defined
  by
  $$\tilde{ \pi}_t = \begin{cases}
    \hat \pi_t & \text{if } t\in [0,1) \\
    \pi_{t-1}  & \text{if } t\in [1,\tilde T-1] \\
    \hat \pi_{\tilde T -t} & \text{if } t\in (\tilde T-1,\tilde T] \\
  \end{cases}$$
  and 
  $$\frac {\de \tilde Q}{\de t} = \begin{cases}
    \frac {\de \hat Q}{\de t}  (t)  & \text{if } t\in [0,1) \\
    \frac {\de Q}{\de t}  (t-1)   & \text{if } t\in [1,\tilde T-1] \\
    \frac {\de \hat Q}{\de t}(\tilde T -t) \circ \Upsilon 
    & \text{if } t\in (\tilde T-1,\tilde T] \\
  \end{cases}
  $$
  Let $\tilde q_T = \frac 1{\tilde T} \tilde Q(1).$
  By construction
  $|\tilde q_{\tilde T} - q | \le c/T.$
  By construction and  Lemma \ref{t:ct} 
  $$i_e(\tilde q_T |M_e) \le \frac 1{\tilde T}
  J_{e,\tilde T}(\tilde{\bs \pi}, \tilde{\bs Q})
  \le  \frac 1{\tilde T} J_{e,T}(\bs \pi,\bs Q) + \frac cT.$$
  In view of the subadditivity proven in Step 1 and the lower semicontinuity
  proven in Step 2, 
  by optimising over $(\bs \pi,\bs Q)\in \ms A_{e,T}(q|M_e)$ and taking the
  limit inferior as $T\to +\infty$, we deduce
  the inequality \eqref{eq:muMe}. The reverse inequality is proven by the same argument, and it follows for all $\mu, \nu \in \ms P_e$ and all $q>0$, \begin{equation}
  	i_e(q|\mu)=i_e(q|M_e)=i_e(q|\nu).
  \end{equation}

  \subsubsection* {Step 4. Convexity on $(0,+\infty)$}

  Thanks to step 3, we take $\mu=M_e$, and omit it from the notation. Thanks to lower semicontinuity,
  for each $q_1, q_2 \in (0,+\infty)$
  \begin{equation}
    \label{eq:conv2}
    i_e\Big(\frac {q_1+q_2}2\Big) \le \frac {i_e(q_1)
  + i_e (q_2)}{2}.\end{equation}
  Fix $\mu$, let $T^1_n$ be any diverging sequence and let $(\pi_n^1,Q_n^1)\in
  \ms A_{e,T_n}(q_1|\mu)$ be chosen
  such that
  $$\lim_{n\to +\infty} \frac 1{T_n} J_{e,T_n}(\bs \pi^1_n,\bs Q^1_n) = i_e(q_1).$$
  For  $(\bs \pi_n,\bs Q_n)\in \ms A_{e,2T_n}((q_1+q_2)/2|\mu)$ as the path
  obtained
  by concatenating $(\bs\pi_n^1,\bs Q^1_n)$ and  $(\bs\pi_n^2,\bs Q^2_n)$.
  Then
  $$i_e\Big(\frac {q_1+q_2}2|\mu\Big) \le
  \frac 1{2T_n} \Big( J_{e,T_n}(\bs \pi^1_n,\bs Q^1_n) + J_{e,T_n}(\bs \pi^2_n,\bs Q^2_n)
  \Big).$$
  We deduce \eqref{eq:conv2} by taking the limit $n\to +\infty$.

  \subsubsection*{ Step 5. $i_e$ is continuous on $(0,+\infty)$}

  By convexity it is enough to show that $i_e$ is bounded.
  Recalling that $i_e(q) = i_e(q|M_e)$, by choosing
  $\pi = M_e$ and $Q = \alpha Q^{M_e}$, with 
  $\alpha = q/\bar q_e$ and $\bar q_e$ defined in \eqref{barqe},
  we obtain that
  $$i_e(q) \le  q \log q/\bar q_e - (q-\bar q_e)< +\infty$$

  \subsubsection* {Step 6. 
  $i_e (q) = 0$ for $q\in (0,\bar q_e]$}

  Given $q\in (0,\bar q_e)$, let $e'$ be such that
  $\bar q_{e'} = q$, where $\bar q_e$ is defined in \eqref{barqe}.
  Observe that $i_e(q) =  i_e(q|M_{e'})$. By choosing  the path
  $(M_{e'},Q^{M_{e'}})$ which is in $\ms A_{e,T}(q|M_{e'})$ for any $T>0$,
  we obtain $i_e(q) = 0$.
  This statement gives also the convexity and continuity of $i_e$ in
  $[0,+\infty)$.

\end{proof}

\begin{proof}[Proof of Theorem~\ref{t:gc}.]
  $~$
We collect the proofs of the different parts of equicoercivity and $\Gamma$-convergence.

\subsubsection*{Proof of i) (Equicoercivity)}

  Fix $q>0$. For any $T>0$ and any path $(\bs \pi,\bs Q)$ in $[0,T]$, such that
  $\bs Q(1) = Tq$, by choosing the constant function $\gamma>0$
  in the variational formula for $J_{e,T}(\bs \pi,\bs Q)$
  proven in \cite{BBBC1},
  $$\frac 1T J_{e,T}(\bs \pi,\bs Q) \ge  q\gamma  - (\ee^\gamma -1 ) \frac 1T \bs Q^\pi(1)
  \ge q\gamma - (\ee^\gamma -1 ) C,$$
  where $C$ is a constant depending only on $e$.
  The statement follows.

\subsubsection*{Proof of ii) ($\Gamma$--$\liminf$)}  Fix $q\in [0,+\infty)$, and consider a sequence $q_T\to q$ as $T\to +\infty$.
  Recalling the definition \eqref{rfit} of $\ms J_{e,T}$,
  $$\frac 1T \ms I_{e,T}( q_T|m) =
  \inf_{(\bs \pi,\bs Q): \bs Q(1) = T q_T}
  \left\{  \frac 1T H_e(\pi_0|m) + \frac 1T J_{e,T}(\bs \pi,\bs Q)
  \right\}.$$
  By the goodness of the rate function \eqref{rft}, the infimum
  is achieved for some path $(\pi^T,Q^T)$ with $Q^T(1) = Tq_T$.
  For any $T_0>0$, the Controllability Theorem  \ref{t:cbe} produces a path
  $(\bar{\bs \pi},\bar{\bs Q})\in \ms S_{e,T_0}$
  with $\bar \pi_0  = \pi^T_T$, $\bar \pi_{T_0} = \pi^T_0$ and, provided $T_0$ is chosen so that $qT_0\ge C_1(e)$,
  $\bar{\bs Q}(1) =qT_0$, satisfying $$J_{e,T_0} (\bar{\bs \pi},\bar{\bs Q})\le F_1(e, T_0, qT_0) + H_e(\pi^T_0|M_e).$$ 

  We now let   $(\tilde{\bs \pi},\tilde{\bs Q})$ be the path in $[0,T+T_0]$ given by concatenating  $(\bar{\bs \pi}, \bar{ \bs Q})$ and $(\bs\pi^T,\bs Q^T)$. This produces
  $$J_{e,T+T_0} (\tilde{\bs \pi}, \tilde {\bs Q}) \le
  F_1(e,T_0, qT_0) + H_e(\pi^T(0)|M_e) +  J_{e,T}(\bs \pi^T,\bs Q^T).$$
  By assumption \ref{ass:2},
  there exists a constant depending only on $m$ and $e$
  such that
  $$H_e(\pi^T(0)|M_e) \le c + H_e(\pi^T(0)|m),$$
  then 
  $$\frac 1{T+T_0} J_{e,T+T_0} (\tilde {\bs \pi}, \tilde {\bs Q}) \le
  \frac {c+F_1(e, T_0, qT_0)}{T+T_0} + \frac {T}{T+T_0} 
  \frac 1T \ms I_{e,T}( q_T|m).$$ 
  Note that $(\tilde \pi , \tilde Q) \in \ms A_{e,T+T_0}(\tilde q_T,\pi^T_0)$
  for some $\tilde q_T \to q$ as $T\to +\infty$.
  Recalling \eqref{rfi0}, 
  we conclude by taking the $\liminf$ and
  using the lower semicontinuity of $i_e(q)$ proven
  in Proposition \ref{prop:ie}.

\subsubsection*{Proof of iii) ($\Gamma$--$\limsup$)}
We split into the cases $q>0, q=0$.

\noindent
For $q>0$, we apply Proposition \ref{prop:ie} to see that
 there exists a sequence of paths 
 $(\bs \pi^T,\bs Q^T)$, such that $\pi^T(0)=m=\pi^T(T)$,  $\bs Q^T(1) = qT$
 and
 $$\lim_{T\to +\infty} \frac 1T J_{e,T} (\bs \pi^T, \bs Q^T) = i_e(q).$$
 By choosing the constant sequence $q_T=q$, we deduce
 $$\ms I_{e,T} (q|m) \le \frac 1T \{ H_e(\pi^T(0)|m) + J_{e,T}(\bs \pi^T,\bs Q^T)\}$$
 which yields the statement.

 \noindent
 For $q=0$, recalling the separate definition of $i_e(0)$ in \eqref{rfi0'}, there exists a sequence $q_n\downarrow 0$
 such that
 $i_e(0) = \lim_n i_e(q_n)$. By the previous result, using a 
 diagonal argument we conclude the proof also for $q=0$.
\end{proof}

\begin{proof}[Proof of Theorem~\ref{t:bi}]
  $~$

\subsubsection*{Upper bound.}
We first prove that $i_e\le i_e^+$ given by \eqref{iunder}.  
  Given $\mu \in \ms P_e$ with density $f$,
  consider the path $(\bs \pi,\bs Q)$ with $\pi_t = \mu$ and 
$\de Q=\ee^\gamma \frac 12 \sqrt{ff_*f'f'_*} B\de t\de v \de v_*
\de \omega$. By the symmetry of $Q$, $(\bs \pi,\bs Q)$ satisfies the continuity equation
\eqref{bal}
and, 
tuning $\gamma$ so that  $\ee^\gamma = q / R_4(\mu)$, we achieve
$(\bs \pi,\bs Q) \in \ms A_{e,T}(q|\mu)$.
By direct computation,
$$\frac 1T J_{e,T}(\bs \pi,\bs Q) =
    q \log \frac{q}{R_4(\mu) } - q +R_2(\mu).$$
  We conclude by optimising in $\mu$.

  \subsubsection*{Lower bound}
  We now prove the lower bound given by \eqref{iover}.

  By the variational representation of $J_{e,T}$ proven in \cite{BBBC1},
  for each
  $F\colon [0,T]\times \bb R^{2d}\times \bb R^{2d}\to \bb R$
  continuous, bounded, and satisfying
  $F(t; v, v_*, v',v_*')$ $=F(t;  v_*, v, v',v_*') = F(t;  v, v_*, v'_*,v')$,
  $$\frac 1T J_{e,T}(\bs \pi,\bs Q) \ge \frac 1T \left( \bs Q(F)
    - \bs Q^\pi ( \ee^F -1) \right).$$
  Given a path $(\bs \pi,\bs Q)\in \ms A_{e,T}(q|\mu)$, by choosing
  $F= \gamma + \frac 12 \log \frac {f'f_*'}{ff_*}$, where $f$ is the
  density of $\pi(t)$, we deduce
  $$\frac 1T J_{e,T}(\bs \pi,\bs Q) \ge  \gamma q - \ee^\gamma \frac 1T \int_0^T
  R_4(\pi_t) \de t+
  \frac 1T \int_0^T
  R_2(\pi_t)\de t,$$
  where we used $\pi_0 = \pi_T$, so that, by the chain rule \eqref{eq: chain rule for entropy},
  $Q(\log \frac {f'f'_*}{ff_*} )= 0$. 
  Optimising in $\gamma$,
  we have 
  $$\frac 1T J_{e,T}(\bs \pi,\bs Q) \ge  q \log \frac q{\frac 1T \int_0^T
    R_4(\pi_t) \de t} - q +
  \frac 1T \int_0^TR_2(\pi_t)\de t.$$
  Since $R_2(\pi_t) \ge R_4(\pi_t)$ and $R_4(\pi_t) \le \hat q_e$,
  we deduce
  $$\frac 1T J_{e,T}(\bs \pi,\bs Q) \ge \inf_{0<a<\hat q_e} \left(
  q \log \frac qa -q + a\right) = i_e^-(q) .$$ 

\end{proof}

\section{Second order asymptotics}

We denote by $\ms M$ the  subset of the finite measures $\bs Q$  on
$\bb R^{2d}\times \bb R^{2d}$  that satisfy
$\bs Q(\de v,\de v_*, \de v',\de v_*')=\bs Q(\de v_*,\de v, \de v',\de v_*') =
\bs Q(\de v,\de v_*, \de v'_*,\de v')$.
Given $\pi\in \ms P_e$, set
$$\bs Q^\pi = \frac 12 \pi (\de v) \pi (\de v_*) B(v-v_*,\de \omega) \in
\ms M.$$

Fix $q\in [0,+\infty)$ and a sequence $q_T\to q$ as $T\to\infty$.
Fix also a sequence $(\bs \pi^T,\bs Q^T) \in \ms S_{e,T}$ such that
$\bs Q^T(1) = Tq_T$, and 
$$\limsup_{T\to +\infty} I_{e,T}((\bs \pi^T,\bs Q^T)|M_e) <+\infty.$$
Each $\bs Q^T$ may be written as
$\bs Q^T (\de t) = \de t \, Q^T_t$, with $Q^T_t \in \ms M$,
a.e. in $t\in [0,T]$. Let us introduce the time average associated to the $\bs Q^T$ given by
\begin{equation}\label{eq: time avg}\vartheta_T = \frac 1T \int_0^T \!\de t\, \delta_{\pi^T_t,Q^T_t}\end{equation}
which is a probability measure on $\ms P_e\times \ms M$.

\begin{lemma}
  \label{lemma:thetaT}
Under the hypotheses above on $(\bs \pi^T, \bs Q^T)$, the sequence   $( \vartheta_T )_{T>0}$ is precompact.
  Furthermore, if $\vartheta$ is any cluster point of $\vartheta_T$, then

 \begin{itemize}
  \item[\it i)]  $\int \vartheta(\de \pi, \de Q) \, Q(1) = q$.
  \item[\it ii)] For $\vartheta$ a.e. $(\pi, Q)$, it holds that $Q=Q^\pi = Q\circ \Upsilon$.
    In particular, $\vartheta$ is supported on the set $$\mathcal{G}_e:=\{(M_{e'}, Q^{M_{e'}}): 0\le e'\le e\}$$ of pairs consisting of a Maxwellian of energy at most $e$, and its associated measures $Q^{M_{e'}}$.
  \end{itemize}
\end{lemma}

\begin{proof}
  We start by proving the compactness.
  Since $\ms P_e$ is compact, by Chebyshev's inequality and
Prohorov's theorem, it is enough to show that there exists $\Phi:
\bb R^{4d} \to [0,+\infty)$ with compact level sets such that
\begin{equation}
 \label{eq:thetaTPHI}
\limsup_{T\to +\infty} \int \!\vartheta_T(\de \pi,\de Q)\,
Q(\Phi) <+\infty
\end{equation}
Choosing
$\Phi = \frac 12 \log ( 1 + |v|^2 + |v_*|^2 + |v'|^2 + |v'_*|^2)$ the condition $\pi\in \ms P_e$ implies
$$\sup_{\pi\in \ms P_e} Q^\pi(\ee^\Phi) < +\infty.$$
For any $Q \ll  Q^\pi$, by the Legendre duality 
$$Q(\Phi) \le Q^\pi ( \ee^\Phi -1 )  + \int
\de Q^\pi \left( \frac {\de Q}{\de Q^\pi} \log \frac {\de Q}{\de Q^\pi}
  - \frac {\de Q}{\de Q^\pi} + 1\right) ,$$
so that
$$\int \vartheta_T(\de \pi, \de Q)\, Q(\Phi) \le
\sup_{\pi\in \ms P_e} Q^\pi(\ee^\Phi) + \frac 1T J_{e,T}(\bs \pi^T,\bs Q^T),$$
which concludes the proof of \eqref{eq:thetaTPHI}.

\vskip.3cm
Let $\vartheta$ be a cluster point of $\vartheta_T$, and pick a sequence
of $T$ such that
$\vartheta_T \to \vartheta$.
By definition of $\vartheta_T$,
$$\int \vartheta_T(\de \pi, \de Q) Q(1) = q_T$$
Item {\it i)} follows from this identity and the uniform integrability
given by
\eqref{eq:thetaTPHI}.

By definition
$$\int \vartheta_T(\de \pi, \de Q)
\int
\de Q^\pi \left( \frac {\de Q}{\de Q^\pi} \log \frac {\de Q}{\de Q^\pi}
  - \frac {\de Q}{\de Q^\pi} + 1\right) \le \frac 1T
I_{e,T}((\bs \pi^T,\bs Q^T)|M_e) $$
Using Fatou's lemma to bound the limit of the left-hand side, and since the rate function appearing on the right-hand side is finite by hypothesis,
$$\int \vartheta(\de \pi, \de Q)
\int
\de Q^\pi \left( \frac {\de Q}{\de Q^\pi} \log \frac {\de Q}{\de Q^\pi}
  - \frac {\de Q}{\de Q^\pi} + 1\right) = 0,$$
which implies that $\vartheta$ a.e. $Q=Q^\pi$.
By Proposition~\ref{lemma:rev}
$$\limsup_{T\to +\infty}
I(\bs \pi^T, \bs Q^T \circ \Upsilon) < +\infty.$$
Arguing as before, we deduce that 
$\vartheta$ a.e. $Q\circ \Upsilon =Q^\pi $.
As follows from \cite[\S 3.2]{CIP},
  the probability measures $\pi$ satisfying $Q^{\pi}=\Upsilon\circ Q^{\pi}$ are
  the Maxwellians $M_{e'}: e'\ge 0$. Since the finiteness of $I_{e,T}$ imposes that $\pi^T_t(\zeta_0)\le e'$ for all $t$, it follows by lower semicontinuity that $\pi(\zeta_0)\le e$ for $\vartheta$-almost all $(\pi, Q)$. Together, this proves that $\vartheta$ is supported on $\mathcal{G}_e$, as claimed.
\end{proof}

\begin{proof}[Proof of Theorem~\ref{t:gc2}]
  The equi-coercivity in item (i) follows the analogous statement
in Theorem \ref{t:gc}. 
\subsubsection*{Proof of ii) ($\Gamma$--$\liminf$)} In this step, we will show that, for any sequence $q_T\to q \in [0,\infty)$, it holds that \begin{equation} \label{eq: conclusion of s2} \liminf_{T\to \infty} \mathcal{J}_{e,T}(q_T|M_e)\ge j_e(q) \end{equation} where $j_e$ is defined by \eqref{j=}. We divide into the cases where $q>\bar q_e, q\le \bar{q}_e$.

For $q>\bar{q}_e$,  the fact that the $\Gamma-$liminf in (ii) is infinite for
$q> \bar q_e$ would follow from the first order asymptotic if
we had proven that $i_e(q)>0$ for $q>\bar q_e$.
Since we have proved it only for $q$ large enough, we need a
separate argument.  Fix $q> \bar q_e$ and suppose that
$q_T \to q$; let us assume for a contradiction that 
\begin{equation}\label{eq: contradiction assumption}\liminf_{T\to +\infty} \ms I_{e,T}(q_T|M_e) < +\infty.\end{equation} 
We may therefore choose competitors $(\bs \pi^T, \bs Q^T)\in \mathcal{S}_{e,T}$ with $\bs Q^T(1)=Tq_T$ and $$ I_{e,T}((\bs \pi^T, \bs Q^T)|M_e)<1+\ms I_{e,T}(q_T|M_e). $$ It therefore follows that the liminf of the left-hand side is finite as $T\to \infty$, and we are in the setting of Lemma \ref{lemma:thetaT}. As a result, there exists a sequence $T_n\to \infty$ along which the time averages $\vartheta_{T_n}$ given by \eqref{eq: time avg} converge to some $\vartheta$, satisfying $\int \vartheta(d\pi, dQ)Q(1)=q$ and
with support in $\ms G_e$. On the other hand, the support condition implies that $\int \vartheta(\de \pi, \de Q) \, Q(1) \le  \bar q_e$, which provides a contradiction. We conclude that the original hypothesis \eqref{eq: contradiction assumption} is false. As a result, the liminf appearing in \eqref{eq: contradiction assumption} is infinite for all $q>\bar{q}_e$, which is the conclusion \eqref{eq: conclusion of s2} in the case $q>\bar{q}_e$.

We next consider the case where $\lim_{T\to +\infty} q_T = q \le \bar q_e$.
If there is no sequence along which $\ms{I}_{e,T}(q_T|M_e)$ is bounded, then the conclusion is trivial. Otherwise, we may pick a subsequence and competitors $(\bs \pi^T, \bs Q^T)$ satisfying $\bs Q^T(1)=q_T$ and for which $I_{e,T}(\bs \pi^T,\bs Q^T|M_e)$ is bounded. Taking the average $\vartheta_T$ as in \eqref{eq: time avg}, we may pass to a further subsequence on which $\vartheta_T \to \vartheta$ for some $\vartheta$.
Since  $\int \vartheta(\de \pi, \de Q)\, Q(1) = q$,
$\vartheta$ gives positive probability
to $\{ (M_{e'}, Q^{M_{e'}}), \, e'\in [0,\e(q)]\}$.
As a consequence, there exists a sequence $t=t(T)\le T$, $t\uparrow +\infty$ 
such that $\bs \pi_t^T \to  M_{\tilde e}$ with $\tilde e\le \e(q)$.
By Proposition~\ref{lemma:rev}
$$I_{e,T}((\bs \pi^T,\bs Q^T)|M_e) \ge H_e(\bs \pi_t^T).$$
By lower semicontinuity of $H_e$,
$$\liminf_{T\to +\infty}  H_e(\bs \pi_t^T) \ge
H_e(M_{\tilde e}) \ge H_e(M_{\e(q)}) = j_e(q)$$ and the claim \eqref{eq: conclusion of s2} is proven for $q\le \bar{q}_e$. \\

\subsubsection*{Proof of iii) ($\Gamma$--$\limsup$)}
The $\Gamma-$--$\limsup$ in (iii) is trivial if $q>\bar q_e$. 
For $q\le \bar q_e$ it is enough to choose  $q_T = q$, and  the path
$\bs \pi^T_t = M_{\eps(q)}$,
$\bs Q^T  = \bs Q^{\bs \pi^T}$.
\end{proof}

  %
  

\section*{Acknowledgments}
G. Basile and D. Benedetto have been  supported  by  PRIN
202277WX43 ``Emergence of condensation-like phenomena in interacting
particle systems: kinetic and lattice models'',  founded by the European
Union - Next Generation EU.
 G. Basile and L. Bertini have been co-funded by the European Union (ERC CoG KiLiM, project number 101125162). D. Heydecker is funded by the Royal Commission for the Exhibition of 1851.

\end{document}